\documentclass{amsart}

\usepackage{amssymb,amsfonts}
\usepackage[all,arc]{xy}
\usepackage{enumerate}
\usepackage{pgf}
\usepackage{pgfkeys}
\usepackage{tikz}
\usepackage{xifthen}
\usepackage{tikz-cd}
\usepackage{chngcntr}

\usetikzlibrary{arrows}
\usetikzlibrary{decorations.markings}
\counterwithin{figure}{section}
\usetikzlibrary{calc}
\usetikzlibrary{external}
\newtheorem{thm}{Theorem}[section]
\newtheorem{cor}[thm]{Corollary}
\newtheorem{prop}[thm]{Proposition}
\newtheorem{lem}[thm]{Lemma}

\theoremstyle{definition}
\newtheorem{defn}[thm]{Definition}
\newtheorem{defns}[thm]{Definitions}
\newtheorem{con}[thm]{Construction}

\theoremstyle{remark}
\newtheorem{rem}[thm]{Remark}

\newcommand{\nc}{\newcommand}
\nc{\slr}{\mathrm{SL}(2,\mathbb R)}
\nc{\R}{\mathbb R}
\nc{\C}{\mathbb C}
\nc{\Z}{\mathbb Z}
\nc{\Q}{\mathbb Q}
\nc{\N}{\mathbb N}
\nc{\td}{\sim}
\nc{\mc}{\mathcal}
\nc{\ec}{\mc{EC}(T,x)}
\nc{\ov}{\overline}
\nc{\Arg}{\mathrm{Arg}}
\nc{\Hom}{\mathrm{Hom}}
\nc{\dbz}{\mathrm{d}\bar z}
\nc{\vp}{\varphi}
\nc{\lra}{\xrightarrow{\hspace*{1cm}}}
\nc{\id}{\mathrm{id}}
\nc{\GL}{\mathrm{GL}}
\nc{\mrm}{\mathrm}
\nc{\SL}{\mrm{SL}}

\nc{\hexagon}[2]{ 
  \begin{scope}[shift={(#1,#2)}]
    \draw (0,0) -- (1,0) -- (1.5,.866) -- (1,1.732) -- (0,1.732) -- (-.5,.866) -- cycle;
    \foreach \x in {(0,0),(1,0),(1.5,.866),(1,1.732),(0,1.732),(-.5,.866),(.5,0),(.5,1.732),(-.25,.433),(1.25,1.299),(1.25,.433),(-.25,1.299),(0.5,.866)}
      {\node[draw,circle,inner sep=1pt,fill] at \x {};}
  \end{scope}
}
\nc{\hexagonprime}[2]{ 
  \begin{scope}[shift={(#1,#2)}]
    \draw (0,0) -- (1,0) -- (1.5,.866) -- (1,1.732) -- (0,1.732) -- (-.5,.866) -- cycle;
    \foreach \x in {(.5,0),(.5,1.732),(-.25,.433),(1.25,1.299),(1.25,.433),(-.25,1.299),(0.5,.866)}
      {\node[draw,circle,inner sep=1pt,fill] at \x {};}
  \end{scope}
}
\nc{\hexagonnodots}[2]{ 
  \begin{scope}[shift={(#1,#2)}]
    \draw (0,0) -- (1,0) -- (1.5,.866) -- (1,1.732) -- (0,1.732) -- (-.5,.866) -- cycle;
  \end{scope}
}
\nc{\trihexagon}[2]{ 
  \hexagon{#1}{#2}
  \draw[shift={(#1,#2)}] (0,0) -- (1,1.732);
  \draw[shift={(#1,#2)}] (1,0) -- (0,1.732);
  \draw[shift={(#1,#2)}] (-.5,.866) -- (1.5,.866);
  \draw[shift={(#1,#2)}] (.5,0) -- (.5,1.732);
  \draw[shift={(#1,#2)}] (-.25,.433) -- (1.25,1.299);
  \draw[shift={(#1,#2)}] (1.25,.433) -- (-.25,1.299);
  \fill[opacity=0.3, shift={(#1,#2)}] (.5,0) -- (1,0) -- (.5,.866) -- cycle;
  \fill[opacity=0.3, shift={(#1,#2)}] (1.25,.433) -- (1.5,.866) -- (.5,.866) -- cycle;
  \fill[opacity=0.3, shift={(#1,#2)}] (1.25,1.299) -- (1,1.732) -- (.5,.866) -- cycle;
  \fill[opacity=0.3, shift={(#1,#2)}] (.5,1.732) -- (0,1.732) -- (.5,.866) -- cycle;
  \fill[opacity=0.3, shift={(#1,#2)}] (-.25,1.299) -- (-.5,.866) -- (.5,.866) -- cycle;
  \fill[opacity=0.3, shift={(#1,#2)}] (-.25,.433) -- (0,0) -- (.5,.866) -- cycle;
}
\nc{\trihexagonnodots}[2]{ 
  \hexagonnodots{#1}{#2}
  \draw[shift={(#1,#2)}] (0,0) -- (1,1.732);
  \draw[shift={(#1,#2)}] (1,0) -- (0,1.732);
  \draw[shift={(#1,#2)}] (-.5,.866) -- (1.5,.866);
  \draw[shift={(#1,#2)}] (.5,0) -- (.5,1.732);
  \draw[shift={(#1,#2)}] (-.25,.433) -- (1.25,1.299);
  \draw[shift={(#1,#2)}] (1.25,.433) -- (-.25,1.299);
  \fill[opacity=0.3, shift={(#1,#2)}] (.5,0) -- (1,0) -- (.5,.866) -- cycle;
  \fill[opacity=0.3, shift={(#1,#2)}] (1.25,.433) -- (1.5,.866) -- (.5,.866) -- cycle;
  \fill[opacity=0.3, shift={(#1,#2)}] (1.25,1.299) -- (1,1.732) -- (.5,.866) -- cycle;
  \fill[opacity=0.3, shift={(#1,#2)}] (.5,1.732) -- (0,1.732) -- (.5,.866) -- cycle;
  \fill[opacity=0.3, shift={(#1,#2)}] (-.25,1.299) -- (-.5,.866) -- (.5,.866) -- cycle;
  \fill[opacity=0.3, shift={(#1,#2)}] (-.25,.433) -- (0,0) -- (.5,.866) -- cycle;
}
\nc{\hexagons}[2]{ 
  \foreach \x in {1,...,#1}
    \foreach \y in {1,...,#2}
      {\hexagonprime{1.5*\x}{-.866*\x+1.732*\y}}
}

\numberwithin{equation}{section}
\title{On the Local Theory of Billiards in Polygons}

\author{Alex C. Becker}

\begin{document}

\begin{abstract}
A periodic trajectory on a polygonal billiard table is \emph{stable} if it persists under any sufficiently small perturbation of the table.
It is a standard result that a periodic trajectory on an $n$-gon gives rise in a natural way to a closed path on an $n$-punctured sphere, 
and that the trajectory is stable iff this path is null-homologous.
We present a novel proof of this result in the language of covering spaces, which generalizes to characterize 
the stable trajectories in neighborhoods of a polygon.
Using this, we classify the stable periodic trajectories near the 30-60-90 triangle,
giving a new proof of a result of Schwartz that no neighborhood of the triangle can be covered by a finite union of orbit tiles.
We also extend a result of Hooper and Schwartz that the isosceles Veech triangles $V_n$ admit no periodic trajectories for $n=2^m,m\ge 2$, 
and examine their conjecture that no neighborhood of $V_n$ can be covered by finitely many orbit tiles.
\end{abstract}

\maketitle

\tableofcontents

\section{Introduction}

One of the simplest possible dynamical systems consists of a single billiard ball bouncing around inside a polygon.
Interest in these systems is twofold. First, they are considered model dynamical systems, and so methods developed
to analyze them are often more widely applicable and offer insight into a variety of systems. Second, billiards in polygons often
arise in physical systems. For example, the system of two unequal masses colliding elastically on the unit interval is isomorphic
to a billiard in a right triangle. Additionally, the system of three masses in a circle with fixed center of mass is isomorphic to
a billiard in an acute triangle. Details of these examples can be found in \S 1.2 of a survey by Masur and Tabachnikov
\cite{FlatStruc}.

The study of polygonal billiards dates back at least to J. F. de Tuschis a Fagnano, 
who in 1775 proved that all acute triangles admit periodic trajectories.
We make all the standard mathematical assumptions--the ball has zero radius, there is no friction, and collisions with the edges are
perfectly elastic. Even so, many questions about the dynamics of these systems remain open.
For example, it is unknown whether every polygon, or even every triangle, admits a periodic trajectory.
In the general case, the strongest known result is that every triangle with angles less than $100$ degrees admits a periodic trajectory,
due to Schwartz \cite{Deg1002}.

In contrast, the so-called \emph{rational case} where the polygon has angles that are rational multiples of $\pi$ is fairly well-understood.
A survey of the techniques applied to rational polygons and the results obtained for them thus far can be found in \cite{FlatStruc}. 
One of the strongest such results is that periodic trajectories are dense in the phase space of trajectories \cite{Rational}. 
When the polygon additionally has the \emph{lattice property}, we can even use the work of Veech in \cite{Veech}
to derive precise asymptotics on the number of periodic trajectories of a given length.

A natural generalization is to ask how the dynamics vary as we perturb the polygon, i.e. how billiards behave locally near
a given polygon in the phase space of polygons.
One question is which trajectories on the polygon are stable under these perturbations.
Surprisingly, some polygons admit no stable trajectories, even if they admit many trajectories;
such polygons include right triangles, due to Hooper \cite{RightTri}, and the Veech triangles $V_n$ for $n=2^m,m\ge 2$
due to Hooper and Schwartz \cite{NearlyIsosc}.

However, characterizing stable trajectories on a polygon $P$ alone does not fully determine the local theory of billiards near $P$.
We also want to know about trajectories which exist on polygons arbitrarily close to $P$, but which degenerate to saddle connections
on $P$ itself. To the author's knowledge, the only work on this subject is that of Schwartz, who proved in \cite{Deg1001} that no
neighborhood of the 30-60-90 triangle can be covered by finitely many orbit tiles.
The aim of this paper is to extend results for stable trajectories on $P$ to results for stable trajectories in neighborhoods of $P$.
Lemma~\ref{lem:from stable} characterizes the saddle connections which arise from stable trajectories.
In \S 4 and \S 5, we apply this to triangles for which enumerations of all saddle connections are known in order to understand the
local theory of billiards near them.

With the exception of \S 5, only a background of basic group theory, algebraic topology and real analysis is assumed,
although the reader may benefit from previous exposure to polygonal billiards.
For \S 5, it is assumed that the reader is familiar with the work of Hooper and Schwartz in \cite{RightTri, NearlyIsosc},
and greater familiarity with group theory and homology is necessary.
Each section builds on the previous sections, with the exception of \S 4 and \S 5 which are logically independent of each other.

\subsection{Background} 
We begin by reviewing the basic concepts in polygonal billiards, and giving the precise definitions
we will make use of later. A more thorough introduction to billiards in polygons can be found in \S 1 of 
\cite{FlatStruc} (with a focus on the rational case), or in \S 2 of \cite{Deg1001} (with a focus on the general case).

\begin{defns}
A \emph{trajectory} on a polygon $P$ is a function $f:\R\to P$ with constant derivative except when $f(t)$ lies on (or \emph{strikes}) an edge, 
at which point the derivative is reflected across the line normal to the edge. 
We also require that $f(t)$ is never a vertex.
If $f(0)=f(a)$ and $f'(0)=f'(a)$ for some $a\ne 0$, we say $f$ is \emph{periodic} and normalize so that $a=1$.
\end{defns}

We will assume the edges of $P$ are labeled from $1$ to $n$.
The standard technique for analyzing a trajectory $f:\R\to P$ is to consider the unfolding of $P$ along $f$.

\begin{defn}
Each time $f$ strikes an edge, instead of reflecting the trajectory, we reflect $P$ across the edge and have the trajectory continue in a straight
line. The union of the resulting polygons is the \emph{unfolding}. 
\end{defn}

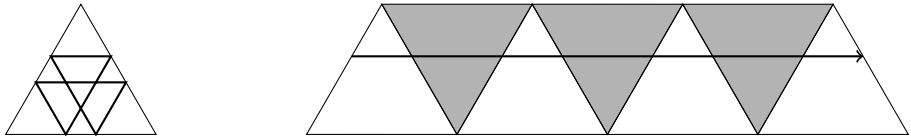
\begin{figure}
\begin{tikzpicture}[scale=2]
\begin{scope}[shift={(-5,0)}]
\draw (0.0000, 0.0000) -- (-0.5000, 0.8660) -- (-1.0000, 0.0000) -- cycle;
\draw[thick, ->] (-.7,.5196) -- (-.3,.5196) -- (-.6,0) -- (-.8,.3464) -- (-.2,.3464) -- (-.4,0) -- cycle;
\end{scope}

\draw (0.0000, 0.0000) -- (-0.5000, 0.8660) -- (-1.0000, 0.0000) -- cycle;
\draw (-1.5000, 0.8660) -- (-0.5000, 0.8660) -- (-1.0000, 0.0000) -- cycle;
\fill[opacity = 0.30] (-1.5000, 0.8660) -- (-0.5000, 0.8660) -- (-1.0000, 0.0000) -- cycle;
\draw (-1.5000, 0.8660) -- (-2.0000, 0.0000) -- (-1.0000, 0.0000) -- cycle;
\draw (-1.5000, 0.8660) -- (-2.0000, 0.0000) -- (-2.5000, 0.8661) -- cycle;
\fill[opacity = 0.30] (-1.5000, 0.8660) -- (-2.0000, 0.0000) -- (-2.5000, 0.8661) -- cycle;
\draw (-3.0000, 0.0000) -- (-2.0000, 0.0000) -- (-2.5000, 0.8661) -- cycle;
\draw (-3.0000, 0.0000) -- (-3.5000, 0.8661) -- (-2.5000, 0.8661) -- cycle;
\fill[opacity = 0.30] (-3.0000, 0.0000) -- (-3.5000, 0.8661) -- (-2.5000, 0.8661) -- cycle;
\draw (-3.0000, 0.0000) -- (-3.5000, 0.8661) -- (-4.0000, 0.0001) -- cycle;
\draw[thick, ->] (-3.7,.5196) -- (-.3,.5196);
\end{tikzpicture}
\caption{The Fagnano orbit $f$ on the equilateral triangle $T$, and the unfolding of $T$ along $f$.}
\label{fig:unfolding}
\end{figure}

This is illustrated in Figure~\ref{fig:unfolding}. Note that we can form the unfolding along an arbitrary path, so long as it does not include
any vertices; we will make heavy use of this later.
Folding these polygons back together gives a covering of $P$ by the unfolding, and a trajectory on $P$ lifts to a straight line on the unfolding.
It is easy to see that a trajectory is periodic iff its lift to the unfolding strikes two copies of $P$ which differ by a translation
in the same place.

\begin{rem}
From here on, we will always work in the unfolding along some trajectory or a more general covering
space. For notational convenience, we call the lift of a trajectory to these spaces a trajectory, and do not distinguish between a trajectory
and its lifts, relying on context to make the underlying space clear.
\end{rem}

Unfolding turns out to be so useful that we will generalize the process in \S 2, and view questions of stability as fundamentally questions about
properties of certain coverings.
One fact made obvious by the unfolding is that periodic trajectories occur in parallel families--translating a periodic trajectory up or down
by $\epsilon$ results in another periodic trajectory for small $\epsilon$. However, for critical values of $\epsilon$, the result will strike a
vertex, producing a saddle connection.

\begin{defn}
A straight line in an unfolding which traverses successive polygons and is periodic but includes one or more vertices is a \emph{saddle connection}.
\end{defn}

This definition is nonstandard. Typically, a saddle connection is defined as a line segment joining a pair of vertices.
Our saddle connections are clearly disjoint unions of such segments, so we call these segments \emph{components} of saddle connections.
Saddle connections are crucial to understanding the local theory of billiards on polygons, 
because they arise as the limiting case of trajectories when we apply Construction~\ref{con:new trajectory}.

We can also use unfoldings to formalize the notion of having the ``same'' periodic trajectory (or closed path) on different polygons, up to translation.

\begin{con}\label{con:new trajectory}
Let $f$ be a closed path on an $n$-gon $P$, $U$ the unfolding of $P$ along $f$ embedded in $\C$ such that $f(0)=0$, 
and $Q$ another $n$-gon. Let $U'$ be the set obtained by replacing each copy of $P$ in $U$ with $Q$.
Let $e$ be the edge in $U$ containing $f(1)$, parametrized by $[0,1]$ in the obvious way, and $t$ such that $e(t)=f(1)$.
Let $e'$ be the corresponding edge in $U'$. We have a unique rotation $R$ such that $Rf$ passes through $e'(t)$, and after renormalizing $f(1)=e(t)$.
If the copies of $Q$ containing $Rf(0)$ and $Rf(1)$ differ by translation and $Q$ is sufficiently close to $P$,
then some translate $g$ of $Rf$ is contained in $U'$ and $U'$ is the unfolding of $Q$ along $g$.
\end{con}

This construction is more complicated than typical, 
but extends to arbitrary closed paths rather than just trajectories, which is crucial for our project.
Intuitively, we replace the copies of $P$ in the unfolding along $f$ with copies of $Q$, and then rotate the path so that it remains closed.
Note that if $f$ is a trajectory, $g$ is as well.
When discussing trajectories on different polygons, we will often apply Construction~\ref{con:new trajectory} implicitly.

\begin{defn}
A periodic trajectory $f$ on an $n$-gon $P$ is \emph{stable} if, for any $n$-gon $Q$ sufficiently near $P$, applying
Construction~\ref{con:new trajectory} to $f$ gives a periodic trajectory on $Q$.
\end{defn}

\begin{defn}
The set of $n$-gons $\mathcal O$ admitting a given periodic trajectory $f$ is the \emph{orbit tile} of $f$.
\end{defn}

\begin{rem}
If $n=3$, we consider $\mathcal O$ a subset of the set of triangles up to similarity, which is parametrized by two angles less than $\pi$.
\end{rem}

Clearly a periodic trajectory is stable iff its orbit tile is open.
The typical strategy for proving the existence of periodic trajectories on general polygons (or at least triangles) 
is to enumerate collections of orbit tiles and show that the tiles cover the set in consideration \cite{NearlyIsosc, Deg1001, Deg1002}.
One important property of orbit tiles requires mention, a proof of which can be found in \S 2.5 of \cite{NearlyIsosc}.

\begin{prop}[\cite{NearlyIsosc}]
If $\mathcal O$ is an open orbit tile, its boundary is piecewise analytic.
\end{prop}

\section{The Stability Lemma}

Much of the theory of stable periodic trajectories is based on the following lemma, which gives a geometric characterization of stability.

\begin{lem}[Stability Lemma]
Let $\mathcal{D}P$ be the surface obtained by identifying two copies of $P$ along the edges, with punctures at the vertices. 
A periodic trajectory on $P$ is stable iff its lift to $\mathcal{D}P$ is null-homologous. 
\end{lem}

This can be proved by combinatorial arguments, as Schwartz did in \cite{Deg1001}, or by using holonomy as Hooper did in \cite{RightTri}.
We present a proof using covering space theory, which makes the geometric connection clear and
will give us a natural way to extend the result to saddle connections. 

First we consider several groups which act on the set of polygons in the plane.

\begin{defns}
Consider a generic $n$-gon $P_g$ with vertices at $(x_i,y_i)$.
Define $G_n$ as the group of affine maps $R_{i_j}\cdots R_{i_1}$ which are obtained by reflecting $P_g$ across its $i_1$th edge,
then reflecting $R_{i_1}P_g$ across its $i_2$th edge, and so on.
Let $D$ denote the derivative map on $G_n$ and define $\tau_n=\ker D$.
Let $\mathrm{eval}_P$ denote the evaluation map on $G_n$ which sends $(x_i,y_i)$ to the $i$th vertex of an $n$-gon $P$
and $\mathrm{eval}_P'$ the evaluation map on $DG_n$.
Define $G_n(P)=\mathrm{eval}_P(G_n)$ and $\tau_n(P)=\ker(\mathrm{eval}_P'\circ D)$.
\end{defns}

This allow us to define the unfolding formally.

\begin{defns}
Puncture $P$ at its vertices. Let 
\[S_g(P)=\coprod\limits_{g\in G_n(P)} gP\Big/\sim\]
where $\sim$ is the equivalence relation obtained by identifying the $i$th edge of $gP$ with the $i$th edge of $R_igP$. 
The \emph{unfolding} of $P$ is $U(P)=S_g(P)/\tau_n(P)$ while the \emph{generic unfolding} of $P$ is $U_g(P)=S_g(P)/\tau_n$,
with $\tau_n$ and $\tau_n(P)$ acting on $S_g(P)$ in the obvious way.
\end{defns}

The unfolding $U(P)$ contains the unfolding of $P$ along any trajectory.
Since copies of $P$ are identified in $U(P)$ iff they differ by a translation, a trajectory on $P$ is periodic iff its lift to $U(P)$ is closed. 
If $f$ lifts further to a closed path in $U_g(P)$, then Construction~\ref{con:new trajectory} 
gives us a closed path in $U_g(Q)$ for any $Q$ sufficiently close to $P$.
Since we have a covering $U_g(Q)\to U(Q)$, this gives us a closed path in $U(Q)$ for all $Q$ sufficiently near $P$, i.e. a periodic trajectory. 
Thus if the lift of $f$ to $U_g(P)$ is closed, $f$ is stable.
In fact, the following lemma shows that for generic $P$ we have $U_g(P)=U(P)$, 
motivating the term ``generic'' and showing that the converse also holds.

\begin{lem}\label{lem:lin ind}
If the angles of $P$ are linearly independent over $\Z$, then $U(P)=U_g(P)$ and hence every periodic trajectory on $P$ is stable.
\end{lem}
\begin{proof}
It suffices to show that $\tau_n(P)=\tau_n$.
Let $\alpha_i$ be the angles of $P$. If $\tau_n(P)\ne \tau_n$, we have some $g\in G_n$ such that $\mathrm{eval}'(Dg)=\id$ but $Dg\ne \id$.
Then $Dg$ cannot be a reflection, so it must be a rotation, hence we can write $Dg$ as a product of rotations of the form
$D(R_iR_{i+1})^{\pm 1}$, which are rotations by $\pm 2\alpha_i$. Since $\mathrm{eval}'(Dg)=0$ we have a linear combination of these rotations
which is trivial, hence by independence $Dg$ is the trivial rotation, a contradiction.
\end{proof}

\begin{cor}
A trajectory on $P$ is stable iff its lift to $U_g(P)$ is closed.
\end{cor}

A closer analysis also gives us the following well-known result.

\begin{cor}\label{cor:not open}
If $\mathcal O$ is not open, the set of angles of polygons in $\mathcal O$ satisfies some integer relation.
\end{cor}
\begin{proof}
Let $P_0\in \mathcal O$ and let $f$ be the trajectory which defines $\mathcal O$. 
Then $f$ must lift to a closed path on $U(P_0)$, so the sequence of reflections
across edges along a period of $f$ lies in $\tau_n(P_0)$, but not in $\tau_n$ as $f$ is not stable.
This sequence must have even length, so we can write it as a product of rotations by $\pm 2\alpha_i$, which must satisfy some integer relation.
For $f$ to be closed on $U(P)$, this sequence of reflections must also lie in $\tau_n(P)$, so this relation must also be satisfied by $P$.
\end{proof}

To relate this to the Stability Lemma, we consider the covering $p:U_g(P)\to \mc DP$ obtained by letting $H_n\subset G_n$ be the (index-$2$) subgroup of orientation-preserving affine maps
and identifying $\mc DP$ with $U_g(P)/H_n$.

\begin{proof}[Proof of Stability Lemma]
Recall that $H_1(\mc DP,\Z)$ is the abelianization of $\pi_1(\mc DP)$, so a closed curve is null-homologous iff its homotopy class
is in the commutator subgroup of $\pi_1(\mc DP)$. Thus in order to prove the lemma it suffices to show
that the image of the induced map $p^*:\pi_1(U_g(P))\to \pi_1(\mc DP)$ is $[\pi_1(\mc DP),\pi_1(\mc PD)]$. Since $S_g(P)$ is simply connected, 
by basic covering space theory the following diagram commutes:
\tikzexternaldisable
\[\begin{tikzcd}
    \pi_1(U_g(P)) \rar{p^*} \dar & \pi_1(\mc DP) \dar\\
    \tau_n\rar{\iota} & H_n
\end{tikzcd}\]
\tikzexternalenable
where the vertical arrows are isomorphisms and $\iota:\tau_n\to H_n$ is the inclusion map. Since the commutator subgroup is a verbal subgroup, it suffices to show that $\tau_n=[H_n,H_n]$.
The inclusion $[H_n,H_n]\subseteq \tau_n$ follows from the fact that $H_n/\tau_n$ is the group of linear parts of elements in $H_n$, which are rotations in the plane and hence commute.
For the reverse inclusion, suppose $[H_n,H_n]\subsetneq \tau_n$, so $H_n/[H_n,H_n]$ is a proper quotient of $H_n/\tau_n$. Note that the elements $R_1R_i$ for $i\ge 2$ generate $H_n$, as $H_n$ consists
of the words in $R_1,\ldots,R_n$ of even length and any $R_iR_j$ with $i,j\ne 1$ is given by $R_iR_j=(R_1R_i)^{-1}(R_1R_j)$. If we let $\theta_i$ be the angle of the $i$th edge of a
generic $n$-gon, then the only linear relation over $\Q$ between the $\theta_i$ and $\pi$ is $\theta_1+\cdots+\theta_n=(n-2)\pi$. Since the image of $R_1R_i$ in $H_n/\tau_n$
is a rotation by $\theta_1-\theta_i$ and the set $\theta_1-\theta_2,\ldots,\theta_1-\theta_n,\pi$ is linearly independent over $\Q$, $H_n/\tau_n$ is a free abelian group on $n-1$ generators, i.e.
$H_n/\tau_n\cong \Z^{n-1}$. But $H_n/[H_n,H_n]\cong H_1(\mc DP,\Z)$, and since $\mc DP$ is an $n$-punctured sphere this is isomorphic to $\Z^{n-1}$. Since $\Z^{n-1}$ is not isomorphic to a proper quotient
of itself, we have a contradiction. Thus $[H_n,H_n]=\tau_n$ so the image of $p^*$ is $[\pi_1(\mc DP),\pi_1(\mc DP)]$ as desired.
\end{proof}

\section{Stability and Saddle Connections}

In order to fully characterize the behavior of periodic trajectories near a polygon $P$, we need to study trajectories which exist on polygons arbitrarily close to $P$ but not
on $P$ itself. In these cases, Construction~\ref{con:new trajectory} gives rise to saddle connections on $P$.

Note that neither Construction~\ref{con:new trajectory} nor the proof of the Stability Lemma require $f$ to be a trajectory; in fact they apply to any closed path.
Thus, if we deform a saddle connection by ``going around'' the vertices with semicircles, we can determine 
whether it arises from a stable periodic trajectory on a nearby polygon.

\begin{con}\label{con:deform}
Let $f$ be a stable periodic trajectory on a polygon $P$. 
If $Q$ is a polygon such that applying Construction~\ref{con:new trajectory} to $f$ gives rise to a saddle connection $g$ on $Q$, let
$t_1,\ldots,t_k$ be the times at which $g$ strikes a vertex. For each $1\leq i\leq k$, if the corresponding vertex of $P$ lies to the left of $f$
then we define $\mathring g(t)$ on the interval $[t_i-\epsilon,t_i+\epsilon]$ to be a counter-clockwise semicircle from 
$g(t_i-\epsilon)$ to $g(t_i+\epsilon)$.
If it lies to the right of $f$, we use a clockwise semicircle instead. For all other values of $t$, 
define $\mathring g(t)=g(t)$.
\end{con}

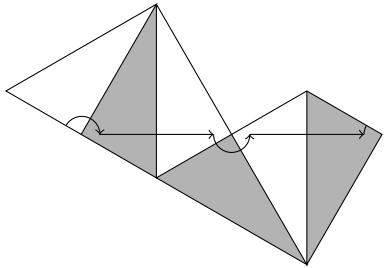
\begin{figure}
\begin{tikzpicture}[scale=2]
\draw (0,0) -- (.5,.866) -- (-.5,.289) -- (.5,-.289) -- (.5,.866);
\fill[opacity = .3] (0,0) -- (.5,-.289) -- (.5,.866) -- cycle;
\draw (.5,.866) -- (1,0) -- (.5,-.289) -- (1.5,-.866) -- (1,0);
\fill[opacity = .3] (.5,-.289) -- (1,0) -- (1.5,-.866) -- cycle;
\draw (1.5,.289) -- (1.5,-.866) -- (2,0) -- (1.5,.289) -- (1,0);
\fill[opacity = .3] (1.5,.289) -- (1.5,-.866) -- (2,0) -- cycle;
\draw[->] (-.1,.06) arc (150:0:.12cm);
\draw[->] (.12,0) -- (.88,0);
\draw[->] (.88,0) arc (180:360:.12cm);
\draw[->] (1.12,0) -- (1.88,0);
\draw (1.88,0) arc (180:150:.12cm);
\end{tikzpicture}
\caption{Construction~\ref{con:deform} applied to the Fagnano orbit on the 30-60-90 triangle.}
\label{fig:deform}
\end{figure}

While $\mathring g$ depends on the choice of trajectory $f$, the corresponding trajectory is usually clear from context, 
so our notation suppresses this dependence.

\begin{rem}
When applying Construction~\ref{con:new trajectory} to a trajectory $f$ gives rise to a saddle connection $g$, 
$g$ often traces out its period several times over the period of $f$. When Construction~\ref{con:deform} is applied this discrepancy disappears,
as different periods of $g$ may have semicircles introduced in different directions.
\end{rem} 

Note that when the vertex is integrable, the other choice of semicircle would result in a different closed path,
while if the vertex is a cone point the other choice would not result in a closed path at all.
The choice of clockwise or counter-clockwise semicircles forces $\mathring g$ to travel along the same sequence of polygons as $f$.
Figure~\ref{fig:deform} illustrates this construction applied to the Fagnano orbit.
Note that while the Fagnano orbit degenerates to
a saddle connection on any right triangle, applying Construction~\ref{con:deform} allows us to avoid the vertices.

Combining Construction~\ref{con:deform} with the Stability Lemma, we get the following necessary condition for a saddle connection to be the result of applying Construction~\ref{con:new trajectory}
to a stable periodic trajectory.

\begin{lem}\label{lem:from stable}
Let $f$ be a periodic trajectory, and $g$ a saddle connection obtained from $f$ by applying Construction~\ref{con:new trajectory}. 
Then $f$ is stable iff $p(\mathring g)$ is null-homologous.
\end{lem}

The major limitation of Lemma~\ref{lem:from stable} is that it does not give us a way to determine whether a
saddle connection arises from a trajectory via Construction~\ref{con:new trajectory}, or if so
tell us anything about the trajectory.
However, for certain triangles of interest to us we can give a partial answer to these questions.
These proofs are illustrated in Figure~\ref{fig:restriction}.

\begin{figure}
\begin{tikzpicture}[scale=2]
\draw (0.0000, 0.0000) -- (1.0000, 0.0000) -- (0.5000, 0.2887) -- cycle;
\draw (0.0000, 0.0000) -- (1.0000, 0.0000) -- (0.5000, -0.2887) -- cycle;
\fill[opacity = 0.30] (0.0000, 0.0000) -- (1.0000, 0.0000) -- (0.5000, -0.2887) -- cycle;
\draw (0.5000, -0.8660) -- (1.0000, 0.0000) -- (0.5000, -0.2887) -- cycle;
\draw (0.5000, -0.8660) -- (1.0000, 0.0000) -- (1.0000, -0.5774) -- cycle;
\fill[opacity = 0.30] (0.5000, -0.8660) -- (1.0000, 0.0000) -- (1.0000, -0.5774) -- cycle;
\draw (1.5000, -0.8660) -- (1.0000, 0.0000) -- (1.0000, -0.5774) -- cycle;
\draw (1.5000, -0.8660) -- (1.0000, 0.0000) -- (1.5000, -0.2887) -- cycle;
\fill[opacity = 0.30] (1.5000, -0.8660) -- (1.0000, 0.0000) -- (1.5000, -0.2887) -- cycle;
\draw (2.0000, -0.0000) -- (1.0000, 0.0000) -- (1.5000, -0.2887) -- cycle;
\draw (2.0000, -0.0000) -- (1.0000, 0.0000) -- (1.5000, 0.2887) -- cycle;
\fill[opacity = 0.30] (2.0000, -0.0000) -- (1.0000, 0.0000) -- (1.5000, 0.2887) -- cycle;
\draw (2.0000, -0.0000) -- (1.5000, 0.8660) -- (1.5000, 0.2887) -- cycle;
\draw (2.0000, -0.0000) -- (1.5000, 0.8660) -- (2.0000, 0.5774) -- cycle;
\fill[opacity = 0.30] (2.0000, -0.0000) -- (1.5000, 0.8660) -- (2.0000, 0.5774) -- cycle;
\draw (2.0000, -0.0000) -- (2.5000, 0.8660) -- (2.0000, 0.5774) -- cycle;
\draw (2.0000, -0.0000) -- (2.5000, 0.8660) -- (2.5000, 0.2887) -- cycle;
\fill[opacity = 0.30] (2.0000, -0.0000) -- (2.5000, 0.8660) -- (2.5000, 0.2887) -- cycle;
\draw (2.0000, -0.0000) -- (3.0000, -0.0000) -- (2.5000, 0.2887) -- cycle;
\draw (2.5000, 0.8660) -- (3.0000, -0.0000) -- (2.5000, 0.2887) -- cycle;
\fill[opacity = 0.30] (2.5000, 0.8660) -- (3.0000, -0.0000) -- (2.5000, 0.2887) -- cycle;
\draw (2.5000, 0.8660) -- (3.0000, -0.0000) -- (3.0000, 0.5774) -- cycle;
\draw (3.5000, 0.8660) -- (3.0000, -0.0000) -- (3.0000, 0.5774) -- cycle;
\fill[opacity = 0.30] (3.5000, 0.8660) -- (3.0000, -0.0000) -- (3.0000, 0.5774) -- cycle;
\draw (3.5000, 0.8660) -- (3.0000, -0.0000) -- (3.5000, 0.2887) -- cycle;
\draw (4.0000, 0.0000) -- (3.0000, -0.0000) -- (3.5000, 0.2887) -- cycle;
\fill[opacity = 0.30] (4.0000, 0.0000) -- (3.0000, -0.0000) -- (3.5000, 0.2887) -- cycle;
\draw (4.0000, 0.0000) -- (3.0000, -0.0000) -- (3.5000, -0.2887) -- cycle;
\draw (4.0000, 0.0000) -- (3.5000, -0.8660) -- (3.5000, -0.2887) -- cycle;
\fill[opacity = 0.30] (4.0000, 0.0000) -- (3.5000, -0.8660) -- (3.5000, -0.2887) -- cycle;
\draw (4.0000, 0.0000) -- (3.5000, -0.8660) -- (4.0000, -0.5774) -- cycle;
\draw (4.0000, 0.0000) -- (4.5000, -0.8660) -- (4.0000, -0.5774) -- cycle;
\fill[opacity = 0.30] (4.0000, 0.0000) -- (4.5000, -0.8660) -- (4.0000, -0.5774) -- cycle;
\draw (4.0000, 0.0000) -- (4.5000, -0.8660) -- (4.5000, -0.2887) -- cycle;
\draw (4.0000, 0.0000) -- (5.0000, 0.0000) -- (4.5000, -0.2887) -- cycle;
\fill[opacity = 0.30] (4.0000, 0.0000) -- (5.0000, 0.0000) -- (4.5000, -0.2887) -- cycle;
\draw (4.0000, 0.0000) -- (5.0000, 0.0000) -- (4.5000, 0.2887) -- cycle;

\draw[very thick] (1, 0) -- (2, 0) node[midway, below, font=\tiny, inner sep=2pt] {$L_1$};
\draw[very thick] (2, 0) -- (3, 0) node[midway, above, font=\tiny, inner sep=2pt] {$L_2$};
\draw[very thick] (3, 0) -- (4, 0) node[midway, below, font=\tiny, inner sep=2pt] {$L_3$};
\node[font=\small] at (1,.08) {$v_1$}; 
\node[font=\small] at (2,-.08) {$v_2$}; 
\node[font=\small] at (3,-.08) {$v_3$}; 
\node[font=\small] at (4,.08) {$v_4$}; 
\end{tikzpicture}
\caption{An illustration of Lemma~\ref{lem:isos restriction}. 
The points $v_1$ and $v_4$ fall below $v_2$ and $v_3$ as we increase the isosceles angle.
Bisecting the triangles gives an illustration of Lemma~\ref{lem:right restriction}.}
\label{fig:restriction}
\end{figure}
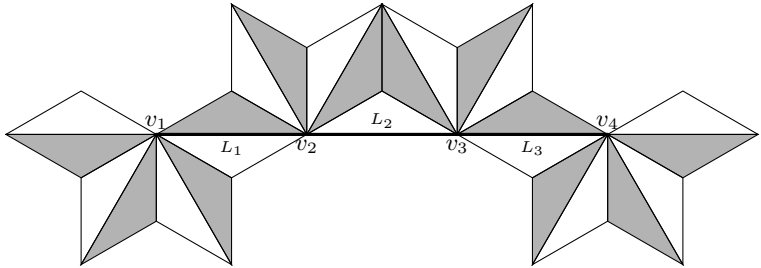

\begin{lem}\label{lem:right restriction}
Let $f$ be a periodic trajectory with orbit tile $\mathcal O$ and $T$ a right triangle.
Suppose applying Construction~\ref{con:new trajectory} to $f$ gives a saddle connection $g$ on $T$ with period more than one,
such that
\begin{itemize}
\item $g$ strikes the vertex with angle $\frac{\pi}{4k+2}\le \alpha\le \frac{\pi}{4k}$ exactly once in its period,
\item the only other vertex $g$ strikes, if any, is the right-angle vertex, and
\item $\mathring g$ includes semicircles in both directions around the vertex with angle $\alpha$.
\end{itemize}
Then for sufficiently small $\epsilon,\delta>0$, 
$\mathcal O\cap B_\delta(T)$ is disjoint from the ray emanating from $T$ with direction $\frac{7\pi}{4}\pm \epsilon$.
\end{lem}
\begin{proof}
Consider the unfolding of $T$ along $f$, embedded in $\C$ such that $f$ travels from left to right.
Let $v_1$ be one of the vertices with angle $\alpha$ such that $\mathring g$ travels counter-clockwise around $v_1$,
while $\mathring g$ travels clockwise around the same vertex $v_2$ in the next period of $g$.
Let $v_3,\ldots,v_j$ be the same vertex in the successive periods of $g$, with $v_j$ the first that $\mathring g$ travels counter-clockwise around.
Let $L_i$ be the line segment between $v_i$ and $v_{i+1}$.

Consider how the unfolding along $f$ changes as we increase the angle $\alpha$ while keeping the right angle constant, i.e. by moving
along the ray emanating from $T$ with direction $\frac{7\pi}{4}$. The unfolding about the $v_i$ consists of $2k+1$ isosceles triangles obtained
from gluing two copies of $T$ 
($f$ might only pass through one of the copies of $T$ in the first and last isosceles triangle, but adding the extra
triangles does not interfere). As $\alpha$ increases, the angle between the first and last isosceles triangle does as well.
Consider the subset of the unfolding starting from the last isosceles triangle around $v_i$ and ending at the first isosceles triangle around
$v_{i+1}$. This is the unfolding along a period of $g$ except for the first and last $k$ isosceles triangles.
Since it can only vary by the direction taken around a right-angle vertex, which is irrelevant
when the right angle is held constant, it evolves identically for all $i$. 
Thus $L_i$ and $L_{i+1}$ will be the same length and their relative direction is determined solely by the unfolding around $v_{i+1}$. 
For $i=2,\ldots,k-1$ the unfolding around $v_i$ wraps above $v_i$, so $L_{i-1}$ and $L_{i}$ fold downwards towards each other.
Thus the line from $v_1$ to $v_j$ passes below $v_i$ for $i=2,\ldots,j-1$. But $f$ passes below $v_1$ and $v_j$ and above $v_i$ for
$i=2,\ldots,j-1$, so $f$ cannot exist on the triangle when we increase $\alpha$.

Since the positions of the $v_i$ are smooth functions of the triangle, for rays sufficiently near the ray in the $\frac{7\pi}{4}$ direction
the result must hold as well.
\end{proof}

Given a trajectory or saddle connection on an isosceles triangle,
quotienting by the reflective symmetry produces a trajectory or saddle connection on the corresponding right triangle.
As a result, the following lemma holds for isosceles triangles.

\begin{lem}\label{lem:isos restriction}
Let $f$ be a periodic trajectory with orbit tile $\mathcal O$ and $T$ an isosceles triangle.
Suppose applying Construction~\ref{con:new trajectory} to $f$ gives a saddle connection $g$ on $T$.
If the corresponding saddle connection on a right triangle satisfies the hypothesis of Lemma~\ref{lem:right restriction},
then for sufficiently small $\epsilon,\delta>0$, 
$\mathcal O\cap B_\delta(T)$ is disjoint from the ray emanating from $T$ with direction $\frac{\pi}{4}\pm \epsilon$.
\end{lem}

\section{Near the 30-60-90 Triangle}

To interpret a neighborhood of the 30-60-90 triangle, we must fix a parametrization of the space of triangles up to similarity. 
We parameterize the triangles by two angles between $0$ and $\pi/2$.
In addition, we label the vertices and edges of the triangle. This is illustrated in Figure~\ref{labeling}.

\begin{figure}
\begin{tikzpicture}[scale=2,font=\small]
\draw (0,0) -- (1,0) node[midway, below] {$e_1$} -- (0,1.732) node[midway, right] {$e_3$} -- (0,0) node[midway, left] {$e_2$};
\draw (0,1.432) arc (-90:-60:.3cm);
\node at (.12,1.353) {$\theta_1$};
\draw (.8,0) arc (180:120:.2cm);
\node at (.74,.15) {$\theta_2$};
\node at (-.04,-.08) {$v_3$};
\node at (1.08,-.08) {$v_2$};
\node at (0,1.8) {$v_1$};
\end{tikzpicture}
\caption{The vertices, edges and angles of a triangle are labeled as shown. The angles $\theta_1=\pi/6$ and $\theta_2=\pi/3$ correspond to the 30-60-90 triangle.}
\label{labeling}
\end{figure}
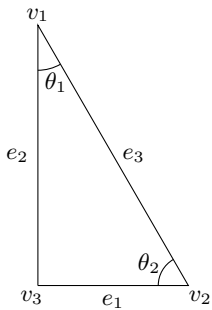

Using Lemma~\ref{lem:from stable}, we classify the orbit tiles which contain the 30-60-90 triangle $T$ in their closure.
As a corollary, we prove:
\begin{thm}\label{no cover}
Let $\mathcal O$ be an open orbit tile such that $T\in \overline{\mathcal O}$.
For any sufficiently small $\epsilon,\delta>0$, the ray emanating from $T$ with angle $\frac{7\pi}{4}-\epsilon$ is disjoint from 
$\mathcal O\cap B_\delta(T)$.
\end{thm}

This is equivalent to Theorem~1.1 of Schwartz in \cite{Deg1001}. 
We begin by enumerating the saddle connections on $T$. 
Note that $U(T)$ is a regular hexagon with opposite edges identified and punctures at the vertices, center, and midpoint of each edge, as illustrated in Figure~\ref{covering}.
Since the regular unit hexagon tiles the plane under translation, we can lift saddle connections to the plane.
If we consider the lift $\tilde f$ of a saddle connection $f$ and let $x=f(0)$, then the lifts of $x$ form a
lattice in the plane generated by two vectors $(0,\sqrt3)^T,(-3/2,\sqrt3/2)^T$, and if we fix some preimage as the origin and take $\tilde f(0)=0$ then $\tilde f$
is completely determined by 
\[\tilde f(1)=n\begin{pmatrix}0\\\sqrt3\end{pmatrix} + m \begin{pmatrix}-3/2\\\sqrt3/2\end{pmatrix}.\]
We can assume $\gcd(n,m)=1$, thus the automorphism group of the lattice acts 
transitively on the set of saddle connections originating from $x$. 
Relative to the basis $(0,\sqrt3)^T,(-3/2,\sqrt3/2)^T$, this automorphism group is $\mathrm{GL}_2(\Z)$.
Fixing the origin as the center of some hexagon and quotienting by the lattice, 
we see that that this is also the automorphism group of $U(T)$, neglecting the punctures.

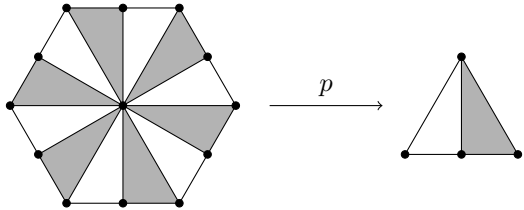
\begin{figure}
\begin{tikzpicture}[scale=1.5]
\trihexagon{0}{0}
\draw[->] (1.8,.866) -- (2.8,.866) node[midway, above] {$p$};
\begin{scope}[shift={(3,.433)}]
\draw (0,0) -- (1,0) -- (.5,.866) -- cycle;
\draw (.5,0) -- (.5,.866);
\fill[opacity=.3] (.5,0) -- (1,0) -- (.5,.866) -- cycle;
\foreach \x in {(0,0),(1,0),(.5,0),(.5,.866)}
  {\node[draw,circle,inner sep=1pt,fill] at \x {};}
\end{scope}
\end{tikzpicture}
\caption{The unfolding $U(T)$ which covers the punctured sphere $\mathcal DT$. The black dots represent punctures. Edge identifications are not shown.}
\label{covering}
\end{figure}

Unfortunately, not every element of $\GL_2(\Z)$ preserves the punctures of $U(T)$.
We overcome this with two modifications. We restrict our attention to the subgroup $G$ which is the kernel of the map $\SL_2(\Z)\to \SL_2(\Z/2\Z)$.
It is easy to verify that $G$ preserves the punctures at $v_1$ and $v_3$, 
thus if we let $U(T)'$ be the surface obtained from $U(T)$ by removing the punctures at
$v_2$ then $G$ acts by automorphisms on $U(T)'$. If we remove the puncture at $v_2$ in $\mathcal DT$ 
to obtain $\mathcal DT'$, we have a six-fold covering
$p':U(T)'\to \mathcal DT'$ analogous to the covering $p:U(T)\to \mathcal DT$.
The inclusion $H_1(\mathcal DT',\Z)\to H_1(\mathcal DT,\Z)$ allows us to apply Lemma~\ref{lem:from stable}.

Observe that any saddle connection striking
$v_1$ 
must also strike
$v_3$ 
by a simple symmetry argument (consider the midpoint of the saddle connection from $v_1$ to itself).
Thus it suffices to consider saddle connections striking $v_2$ or $v_3$. Since the puncture $v_2$ has been removed from $U(T)'$,
the lift to $U(T)'$ of any $\mathring f$ on $T$ is homotopic to one originating from $v_3$.

The action of $G$ on the lattice has two orbits, so every saddle connection originating from $v_i\in T$ 
is the image of one of two saddle connections under
an action of $G$. Thus every saddle connection on $T$ striking $v_1$ or $v_3$
lifts to the image of one of $s_1,s_2,s_3$ or $s_4$ shown in Figure~\ref{fig:cases} under an element of $G$,
up to choice of semicircle direction. The semicircles directions for the $s_i$ are chosen to minimize computations later.

\begin{figure}
\begin{tikzpicture}[scale=2.5]
\hexagonprime{0}{0}
\draw (0.6,0) arc (0:90:.1cm);
\draw (0.5,0.1) -- (0.5,.766);
\draw (0.5,.766) arc (-90:90:.1cm);
\draw (0.5,.966) -- (0.5,1.632) node[midway, right, inner sep=1pt] {$s_1$};
\draw[->] (0.5,1.632) arc (-90:0:.1cm);
\draw (1.15,.433) arc (180:240:.1cm);
\draw (1.15,.433) arc (180:90:.1cm);
\draw (1.25,.533) -- (1.25,1.199) node[midway, left, inner sep=1pt] {$s_2$};
\draw[->] (1.25,1.199) arc (270:120:.1cm);
\draw (-.25,.533) arc (90:120:.1cm);
\draw (-.25,.533) -- (-.25,1.199);
\draw[->] (-.25,1.199) arc (-90:-120:.1cm);
\draw (0.44,0) arc (180:90:.06cm);
\draw (0.5,.06) -- (-.22,1.247) node[midway, right, inner sep=1pt] {$s_3$};
\draw[->] (-.22,1.247) arc (-60:-120:.06cm);
\draw (1.22,.381) arc (240:120:.06cm);
\draw (1.22,.485) -- (0.5,1.672);
\draw[->] (0.5,1.672) arc (-90:-180:.06cm);
\draw (1,0) -- (0.53,.814) node[midway, right, inner sep=1pt] {$s_4$};
\draw (0.53,.814) arc (-60:120:.06cm);
\draw[->] (0.47,.918) -- (0,1.732);
\draw (-.22,.381) arc (-60:120:.06cm);
\draw[->] (-.25,.433) -- (-.5,.866);
\end{tikzpicture}
\caption{The paths $s_1,s_2,s_3$ and $s_4$ on $U(T)'$.}
\label{fig:cases}
\end{figure}
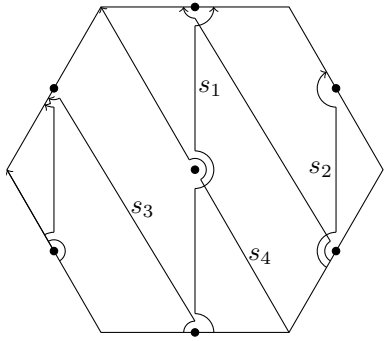

\begin{rem}
A trajectory or saddle connection on $T$ striking only $v_2$ will not lift to the image of any $s_i$, 
but will be homotopic up to choice of semicircles
(in particular with all semicircles chosen to travel the same direction). 
This allows us to treat these cases as special cases of the saddle
connections striking $v_1$ or $v_3$ in our investigation.
\end{rem}

In order to apply Lemma~\ref{lem:from stable}, we consider the action of $G$ on $H_1(U(T)',\Q)$ and consider the
map $H_1(U(T)',\Q)\to H_1(\mathcal DT',\Q)$ induced by the covering.
In order to choose a basis for $H_1(U(T)',\Q)$, we make use of the Mayer-Vietoris sequence for reduced homology.
We break the space $U(T)'$ into two open subsets $A$ and $B$ which intersect in a thin annulus, as shown in Figure~\ref{basis}.
$A$ is a disk with $4$ points removed, so $H_1(A,\Q)\cong \Q^4$ and the curves $p_1,\ldots,p_4$ shown in Figure~\ref{basis}
form a basis for $H_1(A,\Q)$. $B$ deformation retracts to the wedge product of two circles, so $H_1(B,\Q)\cong \Q^2$ and the curves
$c_1,c_2$ form a basis for $H_1(B,\Q)$.

\begin{figure}
\begin{tikzpicture}[scale=2.5]
\hexagons{2}{2}
\path[clip] (1.5,.866) -- (2.5,.866) -- (3,0) -- (4,0) -- (4.5,.866) -- (4,1.732) -- (4.5,2.598) -- (4,3.464) -- (3,3.464) -- (2.5,4.33) -- (1.5,4.33) -- (1,3.464) -- (1.5,2.598) -- (1,1.732) -- cycle;
\draw[->] (3.75,1.05) -- (3.75,2.782) node[right, inner sep=2pt, font=\small] {$c_1$};
\draw[->] (3.75,2.782) -- (2.25,3.649) node[above, inner sep=2pt, font=\small] {$c_2$};
\node at (3.75,3.1) {$B$};
\node at (3,2.4) {$A$};
\draw[->] (3.5,2.498) arc (-90:270:.1cm) node[below, font=\small] {$p_4$};
\draw[->] (2.75,2.931) arc (-90:270:.1cm) node[below, font=\small] {$p_1$};
\draw[->] (2.75,2.065) arc (-90:270:.1cm) node[below, font=\small] {$p_2$};
\draw[->] (3.5,1.832) arc (90:450:.1cm) node[above, font=\small] {$p_3$};
\foreach \x in {-1,0,1}
  \foreach \y in {-1,0,1}
    {\draw[fill, opacity = .3, dashed, shift={(-1.5*\x,.866*\x+1.732*\y)}] (3.65,1.2) -- (3.65,2.732) -- (2.35,3.5) -- (2.35,1.968) -- cycle;}
\end{tikzpicture}
\caption{The decomposition of $U(T)'$ into open subsets $A$ and $B$, and the curves which form the bases of $H_1(A,\Q)$ and $H_1(B,\Q)$, lifted to the plane.}
\label{basis}
\end{figure}
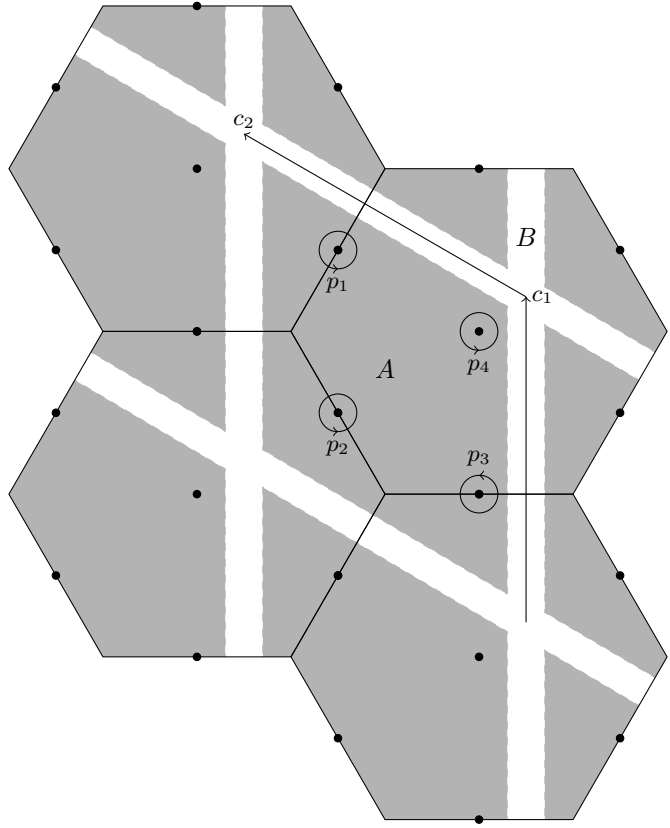

This gives rise to the Mayer-Vietoris sequence
\[0\longrightarrow H_1(A\cap B,\Q)\overset{(i^*,j^*)}{\lra} H_1(A,\Q)\oplus H_1(B,\Q)\overset{k^*-l^*}{\lra} H_1(U(T)',\Q)\longrightarrow 0\]
where the maps $i,j,k,l$ are the usual inclusions. Since $A\cap B$ is an annulus, $H_1(A\cap B,\Q)$ has a basis consisting of a single counter-clockwise loop around the annulus.
The image of this loop under $i^*$ is $p_1+\cdots+p_4$, while its image under $j^*$ is $0$. Thus in $H_1(U(T)',\Q)$, $p_1+\cdots+p_4=0$ and $p_1,p_2,p_3,c_1,c_2$ form a basis.

Since $\mathcal DT'$ is a $2$-punctured sphere, $H_1(\mathcal DT',\Q)\cong \Q$. The map of homologies induced by $p'$ 
is given by ${p'}^*=\begin{pmatrix} 1 & 1 & 1 & -1 & -1\end{pmatrix}$
with respect to the basis consisting of two counterclockwise circles on $\mathcal DT'$.
Since every element of $G$ permutes the loops $p_1,p_2,p_3$ and these have the same image under ${p'}^*$, the action of $G$ and the map ${p'}^*$ factor through the map
\[(x,y,z,u,v)\mapsto (x+y+z,u,v)\]
with ${p'}^*$ becoming $\begin{pmatrix} 1 & -1& -1\end{pmatrix}$ and
\[\begin{pmatrix} a&b\\c&d\end{pmatrix} \;\;\text{acting as}\;\; \begin{pmatrix} 1 & 1-a-c & 1-b-d\\ 0&a&b\\ 0&c&d\end{pmatrix}.\]
These compose to give $\begin{pmatrix} 1 & 1 - 2a-2c & 1-2b-2d\end{pmatrix}$, so the image of a closed path
under an element of $G$ maps down to a null-homologous path on $\mathcal DT'$ iff it is in the kernel of this matrix.

Suppose $f$ is a stable periodic trajectory with orbit tile $\mathcal O$ such that $T\in \ov{\mathcal O}$. Then on $T$, $\mathring f$
is homotopic to the image of some $s_i$ under an element of $G$ repeated $n$ times, up to choice of semicircles.
We examine these cases individually.

\subsection*{Case $s_1$} In this case we prove the following lemma:
\begin{lem}\label{lem:case1}
For sufficiently small $\epsilon,\delta>0$, $\mathcal O\cap B_\delta(T)$ 
is disjoint from the ray emanating from $T$ with angle $\frac{7\pi}{4}\pm \epsilon$.
\end{lem}
\begin{proof}
The curve $s_1$ is homologous to $c_1$ and if repeated $n$ times, clockwise circles around $v_1$ and $v_3$ can be introduced up to $n$ times each.
The resulting curve has homology class
\[\begin{pmatrix}nx\\n\\0\end{pmatrix}\;\;\text{for some}\;x\in [-1,3]\]
since introducing a clockwise circle around $v_1$ contributes $3$ to $p_1+p_2+p_3$ while introducing a clockwise circle around $v_3$ contributes $-1$.
The image of this under an element of $G$ maps down to a curve on $\mathcal DT'$ with homology class $nx+n(1-2a-2c)$. If this is $0$, we have
$0\equiv x+1-2a-2c\equiv x-1\bmod 4$
since $a$ is odd and $c$ is even, thus $x=1$. In particular, $x\ne 0,2$ so not all semicircles can be in the same direction, i.e. $f$ be a
saddle connection on $T$ striking $v_1$ and $v_3$.
Furthermore, if all the semicircles around $v_1$ are clockwise we have $x\le 0$ while if all of them are counter-clockwise we have $x\ge 2$.
Thus we have some semicircles around $v_1$ in each direction, and so the conclusion follows from Lemma~\ref{lem:right restriction}.
\end{proof}

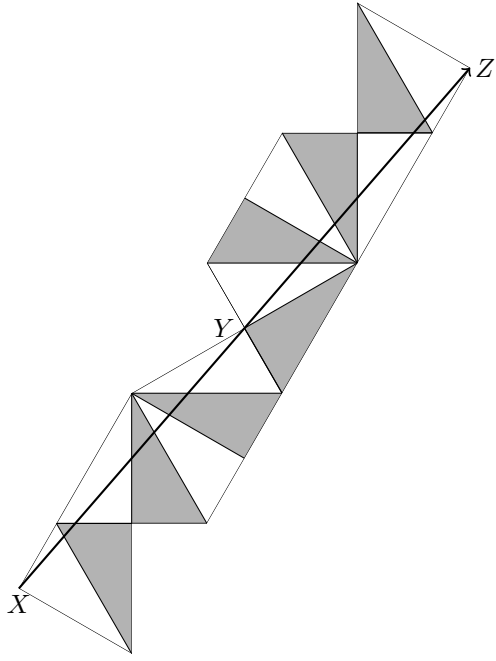
\begin{figure}
\begin{tikzpicture}[scale=2]
\node at (2.75,1.199) {$X$};
\node at (5.85,4.763) {$Z$};
\draw[thick, ->] (2.75,1.299) -- (5.75,4.763) node[midway, left] {$Y$};
\path[clip] (2.75,1.299) -- (3.5,.866) -- (3.5,2*.866) -- (4,2*.866) -- (5.75,4.763) -- (5,5.196) -- (5,5*.866) -- (4.5,5*.866) -- (4,3.464)
 -- (4.25,3.031) -- (3.5,2.598) -- cycle;
\trihexagonnodots{1.5*2}{0}
\trihexagonnodots{1.5*2}{.866*2}
\trihexagonnodots{1.5*3}{.866*3}
\trihexagonnodots{1.5*3}{.866*5}
\end{tikzpicture}
\caption{One repetition of the case $s_2$ for $a=3$.}
\label{fig:case2}
\end{figure}

\subsection*{Case $s_2$} This case includes the Fagnano orbit for acute triangles; in fact, applying Construction~\ref{con:deform} 
to the Fagnano orbit gives $s_2$. For any such $\mathcal O$ we have:
\begin{lem}\label{lem:case2}
For sufficiently small $\delta>0$, $\mathcal O\cap B_\delta(T)$ contains only acute triangles.
\end{lem}
\begin{proof}
If $s_2$ is repeated $n$ times, we can introduce up to $2n$ counter-clockwise circles around $v_3$ and so the resulting curve has homology class
$n(x,1,0)^T$ for some $x\in [0,2]$. As in the previous case, we must have $x=1$, so $f$ strikes $v_3$ on $T$,
and $c=1-a$, so it has homology class $n(1,a,1-a)^T$. 

The unfolding of $T$ along $\mathring f$ is illustrated in Figure~\ref{fig:case2} for $a=3$, 
assuming that the semicircle around $Y$ in $\mathring f$ is counter-clockwise.
As $(\theta_1,\theta_2)$ vary such that the third angle remains right, the points $X,Y,Z$ remain collinear.
As $\theta_1$ decreases with $\theta_2$ fixed, $Y$ falls below the line $\ov{XZ}$, in such a way that the derivative of the distance is nonzero.
Thus for some translate of $f$ to remain within the unfolding,
the semicircles around $X$ and $Z$ must also be counter-clockwise. Since our choice of $Y$ is arbitrary among all copies of $v_3$ in the saddle connection,
it follows that all semicircles must be counter-clockwise, thus $x=2$, contradicting the fact that $x=1$. Since the positions of $X,Y,Z$ are smooth functions
of $(\theta_1,\theta_2)$, by considering the derivative along any direction such that the third angle increases we see that
sufficiently near $T$ the same holds.

In the general case $a=2k+1$ with $k> 0$, the unfolding will include $k$ half-hexagons before $Y$ below the vertex $v_1$, 
followed by $k$ half-hexagons after $Y$ above $v_1$.
The vertices $X,Y,Z$ remain collinear as $(\theta_1,\theta_2)$ vary such that $T$ remains right, as the unfolding is symmetric about $Y$ up to
the direction of the semicircle around $Y$, which is irrelevant for right triangles.
As $\theta_1$ decreases with $\theta_2$ fixed, the half-hexagons below $v_1$ bend to the right, so the line $\overline{XY}$ bends right,
while the half-hexagons above $v_1$ bend to the left, so the line $\overline{YZ}$ bends left. Thus $Y$ falls below $\overline{XZ}$ as in 
the $a=3$ case.
For negative $k$ the half-hexagons below $v_1$ come first, but the trajectory proceeds to the left rather than the right,
so the role of ``left'' and ``right'' are reversed, so the same reasoning applies.
\end{proof}

It seems probable that $\mathcal O$ consists entirely of acute triangles.
However, our fundamentally local methods are not sufficient to establish this.

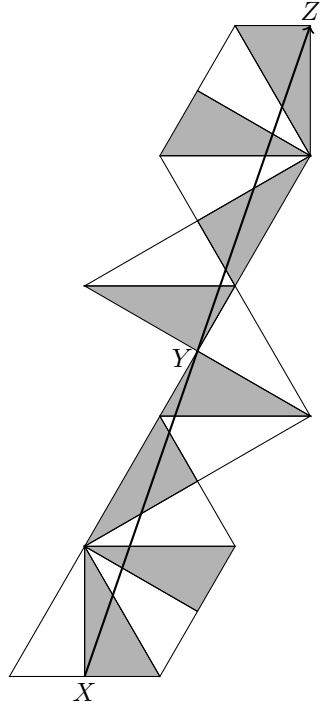
\begin{figure}
\begin{tikzpicture}[scale=2]
\draw (0.0000, 0.0000) -- (-0.5000, -0.8660) -- (0.0000, -0.8660) -- cycle;
\draw (0.0000, 0.0000) -- (0.5000, -0.8660) -- (0.0000, -0.8660) -- cycle;
\fill[opacity = 0.30] (0.0000, 0.0000) -- (0.5000, -0.8660) -- (0.0000, -0.8660) -- cycle;
\draw (0.0000, 0.0000) -- (0.5000, -0.8660) -- (0.7500, -0.4330) -- cycle;
\draw (0.0000, 0.0000) -- (1.0000, 0.0000) -- (0.7500, -0.4330) -- cycle;
\fill[opacity = 0.30] (0.0000, 0.0000) -- (1.0000, 0.0000) -- (0.7500, -0.4330) -- cycle;
\draw (0.0000, 0.0000) -- (1.0000, 0.0000) -- (0.7500, 0.4330) -- cycle;
\draw (0.0000, 0.0000) -- (0.5000, 0.8660) -- (0.7500, 0.4330) -- cycle;
\fill[opacity = 0.30] (0.0000, 0.0000) -- (0.5000, 0.8660) -- (0.7500, 0.4330) -- cycle;
\draw (1.5000, 0.8660) -- (0.5000, 0.8660) -- (0.7500, 0.4330) -- cycle;
\draw (1.5000, 0.8660) -- (0.5000, 0.8660) -- (0.7500, 1.2990) -- cycle;
\fill[opacity = 0.30] (1.5000, 0.8660) -- (0.5000, 0.8660) -- (0.7500, 1.2990) -- cycle;
\draw (1.5000, 0.8660) -- (1.0000, 1.7321) -- (0.7500, 1.2990) -- cycle;
\draw (0.0000, 1.7321) -- (1.0000, 1.7321) -- (0.7500, 1.2990) -- cycle;
\fill[opacity = 0.30] (0.0000, 1.7321) -- (1.0000, 1.7321) -- (0.7500, 1.2990) -- cycle;
\draw (0.0000, 1.7321) -- (1.0000, 1.7321) -- (0.7500, 2.1651) -- cycle;
\draw (1.5000, 2.5981) -- (1.0000, 1.7321) -- (0.7500, 2.1651) -- cycle;
\fill[opacity = 0.30] (1.5000, 2.5981) -- (1.0000, 1.7321) -- (0.7500, 2.1651) -- cycle;
\draw (1.5000, 2.5981) -- (0.5000, 2.5981) -- (0.7500, 2.1651) -- cycle;
\draw (1.5000, 2.5981) -- (0.5000, 2.5981) -- (0.7500, 3.0311) -- cycle;
\fill[opacity = 0.30] (1.5000, 2.5981) -- (0.5000, 2.5981) -- (0.7500, 3.0311) -- cycle;
\draw (1.5000, 2.5981) -- (1.0000, 3.4641) -- (0.7500, 3.0311) -- cycle;
\draw (1.5000, 2.5981) -- (1.0000, 3.4641) -- (1.5000, 3.4641) -- cycle;
\fill[opacity = 0.30] (1.5000, 2.5981) -- (1.0000, 3.4641) -- (1.5000, 3.4641) -- cycle;

\draw[thick, ->] (0, -.866) -- (1.5, 3.4641);
\node at (0, -.966) {$X$};
\node at (.65, 1.25) {$Y$};
\node at (1.5, 3.5641) {$Z$};
\end{tikzpicture}
\caption{One repetition of the case $s_3$ for $a+b=3$.}
\label{fig:case3}
\end{figure}

\subsection*{Case $s_3$} This case is similar to $s_2$, and Lemma~\ref{lem:case2} holds for it as well.
\begin{proof}
The curve $s_3$, repeated $n$ times, admits up to $2n$ counter-clockwise semicircles around $v_3$ and thus has homology class
$n(x+1,1,1)^T$ for some $x\in [0,2]$. The image of this curve under an element of $G$ maps down to a curve with homology class $n(x+3-2(a+b+c+d))$ on
$\mathcal DT'$. If this is $0$ then $0\equiv x+3-2(a+b+c+d)\equiv x+3\bmod 4$, thus $x=1$. Since $x\ne 0,2$, $f$ must strike $v_3$ on $T$. 
From $x=1$ it follows that $a+b+c+d=2$ so $d=2-a-b-c$, and so $\mathring f$ has homology class $n(2,a+b,2-a-b)^T$.

Figure~\ref{fig:case3} shows the unfolding of $T$ along $\mathring f$ for $a+b=3$, assuming
that the semicircle around $Y$ in $\mathring f$ is counter-clockwise. Similar to the previous case, when $(\theta_1,\theta_2)$ vary
such that the third angle remains right, $X,Y,Z$ remain collinear, while when $\theta_1$ decreases with $\theta_2$ fixed,
 $Y$ goes above the line $\ov{XZ}$ in such a way that the derivative of the distance is nonzero.
Thus we conclude that every semicircle must be counter-clockwise, so $x=2$, contradicting the fact that $x=1$.
The same derivative argument extends this result to arbitrary directions such that the third angle increases.

For the general case, again we have that $X,Y,Z$ remain collinear as long as the underlying triangle is right by symmetry.
A similar argument about half-hexagons shows that decreasing $\theta_1$ with $\theta_2$ fixed causes $Y$ to pass above the line $\overline{XZ}$,
but requires more care. The pattern of half-hexagons is
\begin{center}
\begin{tabular}{c|c|c}
$a+b$ & before $Y$ & after $Y$\\
\hline
$1$ & &\\
$3$ & $B$ & $A$\\
$5$ & $BA$ & $BA$ \\
$7$ & $BBA$ & $BAA$\\
$9$ & $BBAA$ & $BBAA$\\
$2k+1$ & $B^{\lceil k/2\rceil}A^{\lfloor k/2\rfloor}$ & $B^{\lfloor k/2\rfloor}A^{\lceil k/2\rceil}$\\
\end{tabular}
\end{center}
where $A$ indicates a half-hexagon above $v_1$ and $B$ a half-hexagon below $v_1$. For $k>0$, the number of half-hexagons below
$v_1$ between $X$ and $Y$ is always at least the number above, and the ones below come first, thus have a larger affect on the line
$\overline{XY}$ since they rotate more of the line. Thus the line $\overline{XY}$ bends left. Symmetrically, the line $\overline{YZ}$ bends right,
so $Y$ moves above $\overline{XZ}$.
For negative $k$ the roles are reversed, but so are the directions ``left'' and ``right'', so the same argument applies.
\end{proof}

\subsection*{Case $s_4$} This is similar to $s_1$, and Lemma~\ref{lem:case1} holds for it as well.
\begin{proof} 
The curve $s_4$, repeated $n$ times, admits up to $n$ clockwise circles around $v_1$ and $v_3$ each, so has homology class $n(x+1,1,1)^T$ for some $x\in [-1,3]$.
As in the previous case we have $x=1$, so not every semicircle around $v_1$ can be in the same direction,
hence the result follows from Lemma~\ref{lem:right restriction}.
\end{proof}

Together, these lemmas prove Theorem~\ref{no cover}.

\section{Near the Veech Triangles}

We can also apply Lemma~\ref{lem:from stable} to examine Hooper and Schwartz's Conjecture~1.6 in \cite{NearlyIsosc} that no neighborhood of 
the isosceles triangle $V_n$ with base angle $\frac{\pi}{2n}$ admits a covering by finitely many orbit tiles for $n=2^m,m\ge 2$.
Experimentally, the main difficulty in covering a neighborhood of $V_n$ seems to be covering a neighborhood of 
the ray emanating from $V_n$ with direction $\frac{\pi}{4}$, i.e. the ray of isosceles triangles with greater base angle.
We derive strong restrictions on any such orbit tile, and show that it must correspond to a trajectory which is poorly behaved near $V_n$.

\begin{thm}\label{thm:veech}
Let $f$ be a periodic trajectory with orbit tile $\mathcal O$ such that $\ov{\mathcal O}$ contains the ray emanating from $V_n$ with
angle $\frac{\pi}{4}$ and length $\delta>0$. If $n=2^m,m\ge 2$, then $f$ degenerates to a saddle connection on $V_n$ which strikes the
isosceles vertex at least $2\lfloor\sqrt{2n}\rfloor$ times over the period of $f$.
\end{thm}

This extends Theorem~1.5 of \cite{NearlyIsosc} that $V_n$ is not contained on the interior of an orbit tile for $n=2^m,m\ge 2$,
which is equivalent to the statement that $f$ degenerates to a saddle connection on $V_n$.
In the case of $V_4$, we can explicitly enumerate the possible counterexamples.

\begin{thm}\label{thm:V_4}
Let $f$ be a periodic trajectory with orbit tile $\mathcal O$ such that $\ov{\mathcal O}$ contains the ray emanating from $V_4$ with
angle $\frac{\pi}{4}$ and length $\delta>0$. On $V_4$, $\mathring f=\pm\phi S_j$ for some $\phi \in \mathrm{Aff}^+(U(V_4))$
where $S_j$ is as in Definition~\ref{def:S_j}.
\end{thm}

It seems implausible that any such trajectory could exist, so Theorem~\ref{thm:V_4} can be taken as strong evidence for Hooper and
Schwartz's conjecture in the case of $V_4$.

We begin by constructing $U(V_n)$ together with a basis for $H_1(U(V_n),\Q)$
and determining the map $p^*:H_1(U(V_n),\Q)\to H_1(\mathcal DV_n)$ induced by the covering, after \S 9 of \cite{NearlyIsosc}.
Cutting $V_n$ along the line of symmetry and reflecting the resulting right triangle across the longer edges gives an $n$-gon.
Doing the same with the other half and gluing the two resulting $n$-gons along their opposite parallel edges gives us $U(V_n)$. This is illustrated
in Figure~\ref{fig:U(V_n)}.

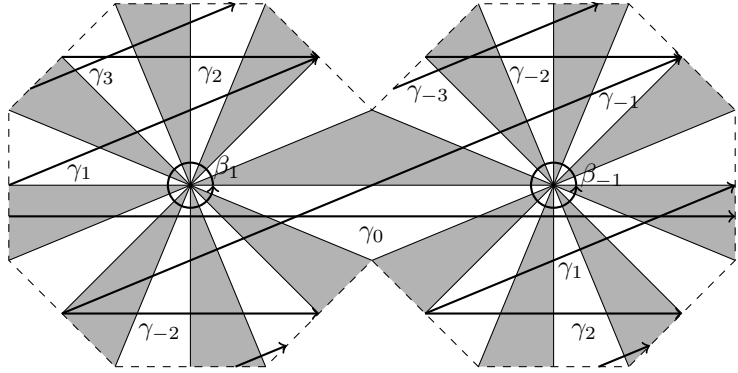
\begin{figure}
\begin{tikzpicture}[scale=2]
\draw[dashed] (0,0) -- (1,0) -- (1.7071,.7071) -- (2.4142,0) -- (3.4142,0) -- (4.1213,.7071) -- (4.1213,1.7071) -- (3.4142,2.4142)
-- (2.4142,2.4142) -- (1.7071,1.7071) -- (1,2.4142) -- (0,2.4142) -- (-.7071,1.7071) -- (-.7071,.7071) -- cycle;
\draw (4.1213,1.2071) -- (-.7071,1.2071);
\draw (0,0) -- (1,2.4142);
\draw (.5,0) -- (.5,2.4142);
\draw (1,0) -- (0,2.4142);
\draw (1.3535,.3535) -- (-.3535,2.0607);
\draw (1.7071,.7071) -- (-.7071,1.7071);
\draw (1.7071,1.7071) -- (-.7071,.7071);
\draw (1.3535,2.0626) -- (-.3535,.3535);
\begin{scope}[shift={(2.4142,0)}]
\draw (0,0) -- (1,2.4142);
\draw (.5,0) -- (.5,2.4142);
\draw (1,0) -- (0,2.4142);
\draw (1.3535,.3535) -- (-.3535,2.0607);
\draw (1.7071,.7071) -- (-.7071,1.7071);
\draw (1.7071,1.7071) -- (-.7071,.7071);
\draw (1.3535,2.0626) -- (-.3535,.3535);
\end{scope}
\fill[opacity = .3] (0.5,1.2071) -- (2.9142,1.2071) -- (1.7071,1.7071) -- cycle;
\fill[opacity = .3] (0.5,0) -- (1,0) -- (.5,1.2071) -- cycle;
\fill[opacity = .3] (1.353,.3535) -- (1.7071,.7071) -- (.5,1.2071) -- cycle;
\fill[opacity = .3] (1.3535,2.0606) -- (1,2.4142) -- (.5,1.2071) -- cycle;
\fill[opacity = .3] (0.5,2.4142) -- (0,2.4142) -- (.5,1.2071) -- cycle;
\fill[opacity = .3] (-.3535,2.0606) -- (-.7071,1.7071) -- (.5,1.2071) -- cycle;
\fill[opacity = .3] (-.7071,1.2071) -- (-.7071,.7071) -- (.5,1.2071) -- cycle;
\fill[opacity = .3] (-.3535,.3535) -- (0,0) -- (.5,1.2071) -- cycle;
\fill[opacity = .3] (1.7071,.7071) -- (2.0606,.3535) -- (2.9142,1.2071) -- cycle;
\fill[opacity = .3] (2.4142,0) -- (2.9142,0) -- (2.9142,1.2071) -- cycle;
\fill[opacity = .3] (3.4142,0) -- (3.7677,.3535) -- (2.9142,1.2071) -- cycle;
\fill[opacity = .3] (4.1213,.7071) -- (4.1213,1.2071) -- (2.9142,1.2071) -- cycle;
\fill[opacity = .3] (4.1213,1.7071) -- (3.7677,2.0606) -- (2.9142,1.2071) -- cycle;
\fill[opacity = .3] (3.4142,2.4142) -- (2.9142,2.4142) -- (2.9142,1.2071) -- cycle;
\fill[opacity = .3] (2.4142,2.4142) -- (2.0606,2.0606) -- (2.9142,1.2071) -- cycle;
\draw[thick, ->] (-.3535,.3535) -- (1.3535,.3535) node[pos=.38, below] {$\gamma_{-2}$};
\draw[thick, ->] (2.0606,2.0606) -- (3.7677,2.0606) node[pos=.41, below] {$\gamma_{-2}$};
\draw[thick, ->] (-.7071,1) -- (4.1213,1) node[midway, below] {$\gamma_0$};
\draw[thick, ->] (-.3535,2.0606) -- (1.3535,2.0606) node[pos=.58, below] {$\gamma_{2}$};
\draw[thick, ->] (2.0606,.3535) -- (3.7677,.3535) node[pos=.62, below] {$\gamma_{2}$};
\draw[thick, ->] (-.5657,1.8485) -- (.8,2.4142) node[pos=.35, below] {$\gamma_{3}$};
\draw[thick, ->] (3.2142,0) -- (3.5556,.1414);
\draw[thick, ->] (-.7071,1.2071) -- (1.3535,2.0606) node[pos=.1, right, inner sep=10pt] {$\gamma_1$};
\draw[thick, ->] (2.0606,.3535) -- (4.1213,1.2071) node[pos=.47, below] {$\gamma_1$};
\draw[thick, ->] (-.3535,.3535) -- (3.7677,2.0606) node[pos=.9, below] {$\gamma_{-1}$};
\draw[thick, ->] (1.8485,1.8485) -- (3.2142,2.4142) node[pos=.17, below, inner sep=4pt] {$\gamma_{-3}$};
\draw[thick, ->] (.8,0) -- (1.1414,.1414);
\draw[thick, ->] (.65,1.2071) arc (0:360:.15cm) node[pos=.15, right] {$\beta_1$};
\draw[thick, ->] (3.0642,1.2071) arc (0:360:.15cm) node[pos=.07, right, inner sep=2pt] {$\beta_{-1}$};
\end{tikzpicture}
\caption{The unfolding $U(V_4)$. Edges on the left octagon are identified with the edges parallel and opposite them on the right octagon.}
\label{fig:U(V_n)}
\end{figure}

If we let $\alpha_1$ and $\alpha_{-1}$ be loops around the punctures at the isosceles angles in $\mathcal DV_n$, then with appropriate
choice of orientation we have
$$p^*(\beta_i)=2n\alpha_i,\quad
p^*(\gamma_i)=\begin{cases}
(i+n)\alpha_1-(i+n)\alpha_2 &\text{if } i<0\\
n\alpha_1 + n\alpha_2 &\text{if } i=0\\
(i-n)\alpha_1-(i-n)\alpha_2 &\text{if } i>0
\end{cases}$$
We let $p^*_1$ and $p^*_{-1}$ denote the composition of $p^*$ with the projections onto the 
$\alpha_1$ and $\alpha_{-1}$ components respectively.

\begin{rem}
Because we build heavily on the work of Hooper and Schwartz in \cite{NearlyIsosc}, it is worth noting two differences in notation:
our $p^*$ is their $\phi$, and our $U(V_n)$ is their $S(V_n)$.
Additionally, we find it convenient to work with homology rather than cohomology,
and it should be noted that $p^*$ denotes the homology map induced by $p$, rather than its cohomological dual.
\end{rem}

The enumeration of saddle connections on $V_n$ is more complicated than for the 30-60-90 triangle.
Again we make use of the affine automorphisms of $U(V_n)$.

\begin{defn}
The \emph{Veech group} $\Gamma(U(P))$ of $U(P)$ is the image in $\mathrm{PSL}_2(\R)$
of the group $\mathrm{Aff}^+(U(P))$ of orientation-preserving affine diffeomorphisms under the derivative map.
\end{defn}

\begin{rem}
This is the same definition used by Veech in \cite{Veech}. It is worth noting that Hooper and Schwartz define $\Gamma(U(V_n))$ 
in \cite{NearlyIsosc} as the corresponding preimage in $\mathrm{SL}_2(\R)$, but this difference is largely immaterial when $4\mid n$ as
we shall see that $-I\in \Gamma(U(V_n))$.
\end{rem}

First we use $\Gamma(U(P))$ to derive the generators for $\mathrm{Aff}^+(U(P))$ used in \cite{NearlyIsosc}.

\begin{lem}\label{lem:generators}
Let $\sigma$ be the involution which exchanges the two $n$-gons in $U(V_n)$, $\tau_e$ the automorphism which acts by simultaneous Dehn twists
through each of the cylinders containing $\gamma_k$ for $k$ even, and $\tau_o$ the corresponding automorphism for $k$ odd.
If $4\mid n$, then $\mathrm{Aff}^+(U(V_n))=\langle \sigma,\tau_e,\tau_o\rangle$.
\end{lem}
\begin{proof}
First we show that the kernel of $\mathrm{Aff}^+(U(V_n))\to \Gamma(U(V_n))$ is $\langle \sigma\rangle$.
Suppose $\phi\in \mathrm{Aff}^+(U(V_n))$ has $D\phi=I$. Note that $\phi$ must preserve the punctures of $U(V_n)$ of a given degree,
in particular the central punctures of each $n$-gon which have degree $1$. Thus $\phi$ either fixes each puncture or exchanges them,
so after possibly applying $\sigma$, $\phi$ fixes each puncture. But then $\phi$ fixes a point, so since $D\phi=I$ it must be the identity.

It remains to show that $\langle D\tau_e,D\tau_o\rangle = \Gamma(U(V_n))$. Note that
$$D\tau_e = \begin{pmatrix} 1 & 2\cot\frac{\pi}{n}\\ 0 & 1\end{pmatrix},\quad
D\tau_o = c_n\begin{pmatrix} 1 & 2\cot\frac{\pi}{n}\\ 0 & 1\end{pmatrix}c_n^{-1}$$
where $c_n$ is a counter-clockwise rotation by $\frac{\pi}{n}$. Theorem~5.8 together with \S 7 of \cite{Veech}
prove that $\Gamma(U(V_n))$ is generated by the three elements
$$\alpha=\begin{pmatrix} 1 & 0\\ 2\cot\frac{\pi}{n} & 1\end{pmatrix},\quad
\alpha'=c_n\begin{pmatrix} 1 & 0\\ 2\cot\frac{\pi}{n} & 1\end{pmatrix}c_n^{-1},\quad c_n^2.$$
Since 
\begin{align*}
c_n^2 &=(D\tau_oD\tau_e)^{n/2+1}\\
\alpha &= c_n^{-n/2}D\tau_e^{-1} c_n^{n/2}\\
\alpha' &= c_n^{-n/2}D\tau_o^{-1} c_n^{n/2}
\end{align*}
and $4\mid n$, the two sets generate the same group.
\end{proof}

We obtain an enumeration of the saddle connections almost identical to the enumeration of trajectories in
Lemma~9.5 of \cite{NearlyIsosc}. First however we require a lemma on the structure of $\Gamma(U(V_n))$, which is well-known
but to the author's knowledge no proof of it exists in literature.

\begin{lem}\label{lem:conjugate}
Any parabolic element of $\Gamma(U(V_n))$ is conjugate to $D\tau_e^k$ or $D\tau_o^k$ for some $k$,
i.e. every maximal parabolic subgroup of $\Gamma(U(V_n))$ is conjugate to $\langle D\tau_e\rangle$ or $\langle D\tau_o\rangle$.
\end{lem}
\begin{proof}
First we show that $\Gamma(U(V_n))$ is a Fuchsian group. This follows from Proposition~2.7 of \cite{Veech}, but can easily be proved directly.
If $x_n\in \Gamma(U(V_n))$ are such that $x_n\to I$, then by taking preimages and
passing to a subsequence we have $y_n\in \mathrm{Aff}^+(U(V_n))$ such that $Dy_n=x_n$ and $y_n\to \mathrm{id}$ or $\sigma$.
By composing with $\sigma$ if necessary we have $y_n\to \mathrm{id}$. But since $y_n$ must take saddle connections of $U(V_n)$ to
saddle connections, clearly $y_n$ must be eventually constant, thus $x_n$ is as well.

In \S 7 of \cite{Veech} it is shown that $\Gamma(U(V_n))$ is a non-cocompact lattice.
Thus by Corollary~4.2.6 in \cite{Fuchsian}, the conjugacy classes of maximal parabolic subgroups are in bijection with the orbits of
vertices of infinity of any Dirichlet region
$$D(p,\Gamma(U(V_n))) = \{z\in\mathbb H: \rho(z,p)\le \rho(T(z),p), \forall T\in \Gamma(U(V_n))\}$$
where $p$ is not fixed by any element of $\Gamma(U(V_n))$. Taking $p$ in the imaginary line and considering the actions of $D\tau_e^{\pm 1}$,
we see that $D(p,\Gamma(U(V_n)))$ is contained in the strip $-\cot\frac{\pi}{n}\le \Re(z)\le \cot\frac{\pi}{n}$,
thus any vertex at infinity must be $\infty$ or lie in the segment $\left[-\cot\frac{\pi}{n},\cot\frac{\pi}{n}\right]$.
Considering the action of $Dc_n^2$ reduces the possibilities to $\infty$ and $\pm\cot\frac{\pi}{n}$. Since 
$D\tau_e\left(-\cot\frac{\pi}{n}\right)=\cot\frac{\pi}{n}$, there are at most two conjugacy classes of maximal parabolic subgroups.
Since $\langle D\tau_e\rangle$ fixes $\infty$ and $\langle D\tau_o\rangle$ fixes $\cot\frac{\pi}{n}$, these classes are as claimed.
\end{proof}

\begin{lem}\label{lem:enumeration}
Any saddle connection on $U(V_n)$ is the image of a saddle connection in the $0$ or $\frac{\pi}{2n}$ direction under some
element of $\mathrm{Aff}^+(U(V_n))$.
\end{lem}
\begin{proof}
Suppose $f$ is a saddle connection on $U(V_n)$.
By Theorem~2.10 and the fact that $\Gamma(U(V_n))$ is a lattice, we have some $\phi\in \mathrm{Aff}^+(U(V_n))$ such that $D\phi\ne I$
and $\phi f=f$. Thus $D\phi$ must be a parabolic element of $\Gamma(U(V_n))$, so we have some $\psi\in \mathrm{Aff}^+(U(V_n))$ and some $k$ such that
$D\psi\phi\psi^{-1}=D\tau_e^k$ or $D\tau_o^k$. Thus $\psi\phi\psi^{-1}$ is one of $\{\pm\tau_e^k,\pm\sigma\tau_e^k,\pm\tau_o^k,\pm\sigma\tau_o^k\}$,
so one of these fixes $\psi f$. Thus $\psi f$ must be in the $0$ or $\frac{\pi}{2n}$ direction.
\end{proof}

Because we restrict our attention to orbit tiles near the line of isosceles triangles, we are able to place powerful restrictions on the
saddle connections which we need to consider. First, we use Lemma~\ref{lem:from stable} to
generalize a result of Hooper in \cite{RightTri} to include saddle connections.

\begin{lem}\label{lem:hooper}
Let $f$ be a saddle connection on a right triangle $T$ with angles that are not rational multiples of $\pi$, 
which arises from a trajectory via Construction~\ref{con:new trajectory}.
Then $f$ strikes the right-angled vertex exactly twice in its period.
Furthermore, if $i$ is the unique involution of $U(T)$ defined in \cite{RightTri}, $f$ is invariant under $i$.
\end{lem}
\begin{proof}
Theorem~3 of \cite{RightTri} proves a similar result for trajectories. The same proof shows that the image of $\mathring f$ in $\mathcal D T$
is not null-homologous, almost without modification. One remark is necessary for Proposition~4: 
if $\mathring f$ and $i(\mathring f)$ intersect, where $i$ is
the unique involution of $U(T)$, then the semicircles in $-i(\mathring f)$ have been both reflected and reversed, thus are in the same direction as
those in $\mathring f$, so $\mathring f=-i(\mathring f)$.

If $f$ strikes the right-angled vertex, let $g$ be $\mathring f$ 
with the modification that no semicircle is introduced around the right-angled vertex.
Then since this vertex is fixed by $i$ we have that $g$ and $-i(g)$ intersect and travel in the same direction,
and by our previous remark they have semicircles in the same direction, thus must agree. Hence $f$ and $-i(f)$ agree as well.
Furthermore $i$ is an orientation-reversing isometry of the image of $g$, so must have exactly two fixed points, thus
$g$ and hence $f$ strike the right-angled vertex exactly twice.
\end{proof}

\begin{rem}
The reader is encouraged to verify that the proof in \cite{RightTri} goes through as claimed. Some notational differences should be noted:
our $U(T)$ is Hooper's $\mathrm{MT}(T)$, our $\mathcal DT$ is his $\mathcal D$, and our $f$ is his $\tilde \gamma$.
\end{rem}

This gives us a corollary for isosceles triangles, similar to Corollary~6 in \cite{RightTri}.

\begin{cor}\label{cor:three cases}
Let $f$ be a saddle connection on an isosceles triangle $T$ which arises via Construction~\ref{con:new trajectory} 
from a stable trajectory with orbit tile $\mathcal O$
such that $\ov{\mathcal O}$ contains an interval of isosceles triangles containing $T$.
Then $f$ strikes the midpoint of the base of $T$ exactly twice in its period, 
and is invariant under the map $i'$ obtained by rotating each copy of $T$ in $U(T)$ by $\pi$ about the midpoint of its base.
\end{cor}
\begin{proof}
It suffices to show that this holds for all isosceles triangles in the interval with irrational base angle, 
as then it holds for the whole interval by density and hence for $T$.
Quotienting by the reflection across the line of symmetry of $T$ gives a saddle connection satisfying the hypothesis of Lemma~\ref{lem:hooper}
which thus strikes the right-angled vertex twice. Thus $f$ strikes the midpoint of the base twice, which is the fixed point of $i'$,
so by the same argument as in Lemma~\ref{lem:hooper}, $f=-i'(f)$.
\end{proof}

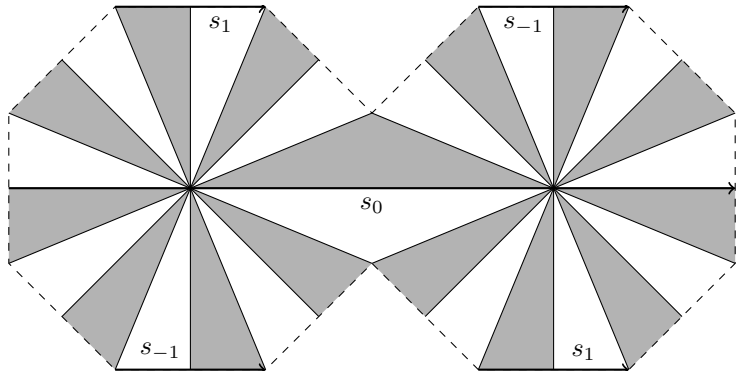
\begin{figure}
\begin{tikzpicture}[scale=2]
\draw[dashed] (0,0) -- (1,0) -- (1.7071,.7071) -- (2.4142,0) -- (3.4142,0) -- (4.1213,.7071) -- (4.1213,1.7071) -- (3.4142,2.4142)
-- (2.4142,2.4142) -- (1.7071,1.7071) -- (1,2.4142) -- (0,2.4142) -- (-.7071,1.7071) -- (-.7071,.7071) -- cycle;
\draw (4.1213,1.2071) -- (-.7071,1.2071);
\draw (0,0) -- (1,2.4142);
\draw (.5,0) -- (.5,2.4142);
\draw (1,0) -- (0,2.4142);
\draw (1.3535,.3535) -- (-.3535,2.0607);
\draw (1.7071,.7071) -- (-.7071,1.7071);
\draw (1.7071,1.7071) -- (-.7071,.7071);
\draw (1.3535,2.0626) -- (-.3535,.3535);
\begin{scope}[shift={(2.4142,0)}]
\draw (0,0) -- (1,2.4142);
\draw (.5,0) -- (.5,2.4142);
\draw (1,0) -- (0,2.4142);
\draw (1.3535,.3535) -- (-.3535,2.0607);
\draw (1.7071,.7071) -- (-.7071,1.7071);
\draw (1.7071,1.7071) -- (-.7071,.7071);
\draw (1.3535,2.0626) -- (-.3535,.3535);
\end{scope}
\fill[opacity = .3] (0.5,1.2071) -- (2.9142,1.2071) -- (1.7071,1.7071) -- cycle;
\fill[opacity = .3] (0.5,0) -- (1,0) -- (.5,1.2071) -- cycle;
\fill[opacity = .3] (1.353,.3535) -- (1.7071,.7071) -- (.5,1.2071) -- cycle;
\fill[opacity = .3] (1.3535,2.0606) -- (1,2.4142) -- (.5,1.2071) -- cycle;
\fill[opacity = .3] (0.5,2.4142) -- (0,2.4142) -- (.5,1.2071) -- cycle;
\fill[opacity = .3] (-.3535,2.0606) -- (-.7071,1.7071) -- (.5,1.2071) -- cycle;
\fill[opacity = .3] (-.7071,1.2071) -- (-.7071,.7071) -- (.5,1.2071) -- cycle;
\fill[opacity = .3] (-.3535,.3535) -- (0,0) -- (.5,1.2071) -- cycle;
\fill[opacity = .3] (1.7071,.7071) -- (2.0606,.3535) -- (2.9142,1.2071) -- cycle;
\fill[opacity = .3] (2.4142,0) -- (2.9142,0) -- (2.9142,1.2071) -- cycle;
\fill[opacity = .3] (3.4142,0) -- (3.7677,.3535) -- (2.9142,1.2071) -- cycle;
\fill[opacity = .3] (4.1213,.7071) -- (4.1213,1.2071) -- (2.9142,1.2071) -- cycle;
\fill[opacity = .3] (4.1213,1.7071) -- (3.7677,2.0606) -- (2.9142,1.2071) -- cycle;
\fill[opacity = .3] (3.4142,2.4142) -- (2.9142,2.4142) -- (2.9142,1.2071) -- cycle;
\fill[opacity = .3] (2.4142,2.4142) -- (2.0606,2.0606) -- (2.9142,1.2071) -- cycle;
\draw[thick, ->] (0,0) -- (1,0) node[pos=.3, above] {$s_{-1}$};
\draw[thick, ->] (2.4142,2.4142) -- (3.4142,2.4142) node[pos=.3, below] {$s_{-1}$};
\draw[thick, ->] (-.7071,1.2071) -- (4.1213,1.2071) node[midway, below] {$s_0$};
\draw[thick, ->] (0,2.4142) -- (1,2.4142) node[pos=.7, below] {$s_1$};
\draw[thick, ->] (2.4142,0) -- (3.4142,0) node[pos=.7, above] {$s_1$};
\end{tikzpicture}
\caption{The saddle connection $s_0$ and saddle connection components $s_{\pm1}$.}
\label{fig:three cases}
\end{figure}

It is easy to check that $\sigma,\tau_e$ and $\tau_o$ all fix the lifts of the base of $V_n$. Thus we only need to be concerned about saddle
connections which contain the image of one of $s_{-1},s_0$ or $s_1$ shown in Figure~\ref{fig:three cases}
under an element of $\langle \sigma,\tau_e,\tau_o\rangle$.

The $s_0$ case is relatively easy to eliminate.

\begin{lem}\label{lem:s_0}
For any $\phi\in \langle \sigma,\tau_e,\tau_o\rangle$, if $f=\phi s_0$ is a saddle connection on $V_n$ which arises from a stable trajectory 
via Construction~\ref{con:new trajectory}, then $\mathring f$ satisfies the hypothesis of Lemma~\ref{lem:isos restriction}.
\end{lem}
\begin{proof}
Clearly any saddle connection containing $s_0$ consists solely of $s_0$.
Note that if a saddle connection $f$ on $U(V_n)$ is identical in both $n$-gons, i.e. $f=\sigma f$, then $\tau_e^{\pm 1} f$
and $\tau_o^{\pm 1}f$ are as well. Thus by induction $\sigma f=f$,
so the image of $f$ on the corresponding right triangle has half the period of $f$, thus it only strikes 
the vertex with angle $\frac{\pi}{2n}$ once in its period.

It remains to show that the semicircles in $\mathring f$ cannot all be in the same direction. 
If $\mathring f$ consists of $s_0$ repeated $k$ times with counter-clockwise semicircles around each vertex,
then $\mathring f$ is homotopic to $k$ repetitions of $\gamma_0$ and $p^*(f) = kn(\alpha_{-1}+\alpha_1)$.
By Proposition~9.7 of \cite{NearlyIsosc} (or a simple check of the actions of the generators), $(p^*_1+p^*_{-1})(f)$ is invariant under
the action of $\langle \sigma,\tau_e,\tau_o\rangle$, so $(p^*_1+p^*_{-1})(g\cdot f)=2kn\ne 0$ for any 
$g\in \langle \sigma,\tau_e,\tau_o\rangle$.
Similarly, if we introduce only clockwise semicircles $(p^*_1+p^*_{-1})(g\cdot f)=-2kn\ne 0$. Thus semicircles
in both directions must be introduced.
\end{proof}

The cases of $s_{-1}$ and $s_1$ are clearly analogous to each other, so we focus our attention on $s_1$.

\begin{figure}
\begin{tikzpicture}[scale=1.5]
\draw[->] (0,0) -- (1,0) node[midway, above] {$L_{-5}$};
\draw[dashed] (1,0) -- (1.866,.5) -- (2.366,1.366);
\draw[dashed] (2.366,2.366) -- (1.866,3.232) -- (1,3.732);
\draw[<-] (1,3.732) -- (0,3.732) node[midway, below] {$L_5$};
\draw[dashed] (0,3.732) -- (-.866,3.232) -- (-1.366,2.366) -- (-1.366,1.366) -- (-.866,.5) -- (0,0);
\draw[->] (-.866,.5) -- (1.866,.5) node[midway, above] {$L_{-3}$};
\draw[->] (-1.366,1.366) -- (2.366,1.366) node[midway, above] {$L_{-1}$};
\fill[opacity = .3] (-.866,.5) -- (1.866,.5) -- (2.366,1.366) -- (-1.366,1.366) -- cycle;
\draw[->] (-1.366,2.366) -- (2.366,2.366) node[midway, below] {$L_1$};
\draw[->] (-.866,3.232) -- (1.866,3.232) node[midway, below] {$L_3$};
\fill[opacity = .3] (-1.366,2.366) -- (2.366,2.366) -- (1.866,3.232) -- (-.866,3.232) -- cycle;

\begin{scope}[shift={(3.732,0)}]
\draw[->] (0,0) -- (1,0) node[midway, above] {$R_5$};
\draw[dashed] (1,0) -- (1.866,.5) -- (2.366,1.366) -- (2.366,2.366) -- (1.866,3.232) -- (1,3.732);
\draw[<-] (1,3.732) -- (0,3.732) node[midway, below] {$R_{-5}$};
\draw[dashed] (0,3.732) -- (-.866,3.232) -- (-1.366,2.366);
\draw[dashed] (-1.366,1.366) -- (-.866,.5) -- (0,0);
\draw[->] (-.866,.5) -- (1.866,.5) node[midway, above] {$R_3$};
\draw[->] (-1.366,1.366) -- (2.366,1.366) node[midway, above] {$R_1$};
\fill[opacity = .3] (-.866,.5) -- (1.866,.5) -- (2.366,1.366) -- (-1.366,1.366) -- cycle;
\draw[->] (-1.366,2.366) -- (2.366,2.366) node[midway, below] {$R_{-1}$};
\draw[->] (-.866,3.232) -- (1.866,3.232) node[midway, below] {$R_{-3}$};
\fill[opacity = .3] (-1.366,2.366) -- (2.366,2.366) -- (1.866,3.232) -- (-.866,3.232) -- cycle;
\end{scope}
\end{tikzpicture}
\caption{Labeling of the components of the horizontal saddle connections on $V_6$. Note that $L_{\pm 5}=R_{\pm 5}$.}
\label{fig:horizontal cylinders}
\end{figure}
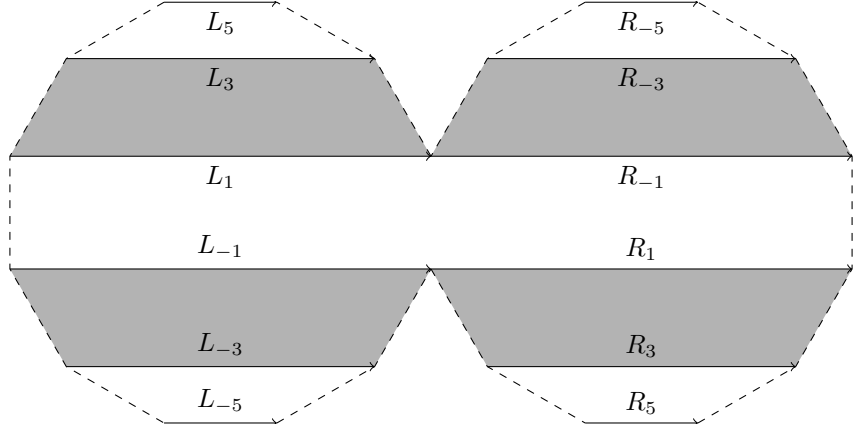

\begin{proof}[Proof of Theorem~\ref{thm:veech}]
By Theorem~1.5 of \cite{NearlyIsosc}, we may assume $f$ degenerates to a saddle connection on $V_n$, and WLOG we may assume it contains $s_1$.
Label the components of the horizontal saddle connections on $U(V_n)$ as shown in Figure~\ref{fig:horizontal cylinders}.
Let $g = \phi^{-1}\mathring f$ for some $\phi\in \langle \sigma,\tau_e,\tau_o\rangle$ be a horizontal saddle connection with semicircles introduced. 
The transitions between the components of $g$, depending on the choice of semicircles, are:
\begin{align*}
L_{\pm(n-1)}=R_{\pm(n-1)} &\to \begin{cases}
L_{\pm(n-3)} &\text{clockwise semicircle}\\
R_{\pm(n-3)} &\text{counter-clockwise semicircle}
\end{cases}\\
L_{k} &\to \begin{cases}
R_{k+2} &\text{clockwise semicircle}\\
R_{k-2} &\text{counter-clockwise semicircle}
\end{cases}\\
R_{k} &\to \begin{cases}
L_{k-2} &\text{clockwise semicircle}\\
L_{k+2} &\text{counter-clockwise semicircle}
\end{cases}
\end{align*}
We will deal with the case $n=4$ separately, so we may assume $n\ge 8$.
If $g$ contains $L_{1-n}$ then it must contain either $L_k$ or $R_k$ for all odd $k$ between $L_{n-1}$ and $L_{1-n}$, and then again
between $L_{1-n}$ and $L_{n-1}$, hence it contains at least $n\ge 2\lfloor \sqrt{2n}\rfloor$ such components and we are done.
Otherwise, if $g$ begins with $L_{n-1}L_{n-3}$ then the remaining components before returning to $L_{n-1}$ must be $L_k$ for
$k\equiv n-3\mod 4$ and $R_k$ for $k\equiv n-1\mod 4$. By Corollary~\ref{cor:three cases}, $g$ contains exactly two copies of $L_{n-1}$.
Since $i'$ interchanges $L_k$ and $R_k$, after $g$ returns to $L_{n-1}$ it travels around semicircles in directions exactly opposite to
the first time. Thus the resulting curve is identical to when $g$ begins with $L_{n-1}R_{n-3}$, up to reparametrization.

The homology class of $g$ is given by 
$$\sum_{k\ge 3} m_k\gamma_{k-1} + m_1(\gamma_0-\beta_1-\beta_{-1}) + \sum_{k\le -1} m_k\gamma_{k+1}$$
where $m_k$ is the number of occurrences of $L_k$ or equivalently of $R_k$.
By Proposition~9.9 of \cite{NearlyIsosc}, for any $\phi\in\langle \sigma,\tau_e,\tau_o\rangle$ we have some odd $r$ such that
\begin{align*}
&(p^*_1-p^*_2)\left(\phi^*\left(\sum_{k\ge 3} m_k\gamma_{k-1} + m_1(\gamma_0-\beta_1-\beta_{-1}) + \sum_{k\le -1} m_k\gamma_{k+1}\right)\right)
\tag{5.1}\\
&\equiv r\left(\sum_{k\ge 3} m_k (k-1+n) + \sum_{k\le -3} m_k (k+1+n)\right)\\
&\equiv -r\left(\sum_{k\ge 3} m_k(n-k+1) +\sum_{k\le -3} m_k(n-k-1)\right)\mod 2n
\end{align*}
If we have at most $x<n$ components, this sum is positive and is maximized when each $m_k$ is $1$, 
except $m_{n-1}$ which is necessarily $2$, which gives 
\begin{align*}
\sum_{k\ge 3} m_k(n-k+1) +\sum_{k\le -3} m_k(n-k-1) &\le 4+\sum_{k=n+1-x}^{n-3} (n-k+1)\\
&=2+(2+4+\cdots+ x - 2)\\
&=2+\left(\frac{x}{2}-1\right)\frac{x}{2}
\end{align*}
and by simple algebra this is less than $2n$ whenever $x<1+\sqrt{8n-7}$.
Thus the expression in (5.1) is nonzero $\bmod\: 2n$ when $x=2\lfloor \sqrt{2n}\rfloor$, 
so $p(\mathring f)$ is not null-homologous, contradicting Lemma~\ref{lem:from stable}.
\end{proof}

In the case of $V_4$, the same analysis gives us a much stronger restriction.

\begin{defn}\label{def:S_j}
$S_j$ is the path which travels along $L_3(L_1R_{-1})^jL_{-3}(L_{-1}R_1)^j$.
\end{defn}

The path $S_1$ is illustrated in Figure~\ref{fig:S_1}.

\begin{figure}
\begin{tikzpicture}[scale=2]
\draw[dashed] (0,0) -- (1,0) -- (1.7071,.7071) -- (2.4142,0) -- (3.4142,0) -- (4.1213,.7071) -- (4.1213,1.7071) -- (3.4142,2.4142)
-- (2.4142,2.4142) -- (1.7071,1.7071) -- (1,2.4142) -- (0,2.4142) -- (-.7071,1.7071) -- (-.7071,.7071) -- cycle;
\draw (4.1213,1.2071) -- (-.7071,1.2071);
\draw (0,0) -- (1,2.4142);
\draw (.5,0) -- (.5,2.4142);
\draw (1,0) -- (0,2.4142);
\draw (1.3535,.3535) -- (-.3535,2.0607);
\draw (1.7071,.7071) -- (-.7071,1.7071);
\draw (1.7071,1.7071) -- (-.7071,.7071);
\draw (1.3535,2.0626) -- (-.3535,.3535);
\begin{scope}[shift={(2.4142,0)}]
\draw (0,0) -- (1,2.4142);
\draw (.5,0) -- (.5,2.4142);
\draw (1,0) -- (0,2.4142);
\draw (1.3535,.3535) -- (-.3535,2.0607);
\draw (1.7071,.7071) -- (-.7071,1.7071);
\draw (1.7071,1.7071) -- (-.7071,.7071);
\draw (1.3535,2.0626) -- (-.3535,.3535);
\end{scope}
\fill[opacity = .3] (0.5,1.2071) -- (2.9142,1.2071) -- (1.7071,1.7071) -- cycle;
\fill[opacity = .3] (0.5,0) -- (1,0) -- (.5,1.2071) -- cycle;
\fill[opacity = .3] (1.353,.3535) -- (1.7071,.7071) -- (.5,1.2071) -- cycle;
\fill[opacity = .3] (1.3535,2.0606) -- (1,2.4142) -- (.5,1.2071) -- cycle;
\fill[opacity = .3] (0.5,2.4142) -- (0,2.4142) -- (.5,1.2071) -- cycle;
\fill[opacity = .3] (-.3535,2.0606) -- (-.7071,1.7071) -- (.5,1.2071) -- cycle;
\fill[opacity = .3] (-.7071,1.2071) -- (-.7071,.7071) -- (.5,1.2071) -- cycle;
\fill[opacity = .3] (-.3535,.3535) -- (0,0) -- (.5,1.2071) -- cycle;
\fill[opacity = .3] (1.7071,.7071) -- (2.0606,.3535) -- (2.9142,1.2071) -- cycle;
\fill[opacity = .3] (2.4142,0) -- (2.9142,0) -- (2.9142,1.2071) -- cycle;
\fill[opacity = .3] (3.4142,0) -- (3.7677,.3535) -- (2.9142,1.2071) -- cycle;
\fill[opacity = .3] (4.1213,.7071) -- (4.1213,1.2071) -- (2.9142,1.2071) -- cycle;
\fill[opacity = .3] (4.1213,1.7071) -- (3.7677,2.0606) -- (2.9142,1.2071) -- cycle;
\fill[opacity = .3] (3.4142,2.4142) -- (2.9142,2.4142) -- (2.9142,1.2071) -- cycle;
\fill[opacity = .3] (2.4142,2.4142) -- (2.0606,2.0606) -- (2.9142,1.2071) -- cycle;
\draw[thick] (.2,2.4142) arc (0:-135:.2cm) node[pos=1, font=\tiny, left] {$6$};
\draw[thick] (.2,2.4142) -- (.8,2.4142) node[pos=1, font=\tiny, above] {$1$};
\draw[thick] (3.2142,0) arc (180:45:.2cm) node[pos=0, font=\tiny, below] {$1$} node[pos=1, font=\tiny, right] {$2$};
\draw[thick] (-.5071,1.7071) arc (0:45:.2cm) node[pos=1, font=\tiny, left] {$2$};
\draw[thick] (-.5071,1.7071) -- (1.5071,1.7071);
\draw[thick] (1.5071,1.7071) arc (180:360:.2cm);
\draw[thick] (1.9071,1.7071) -- (3.9213,1.7071);
\draw[thick] (3.9213,1.7071) arc (180:135:.2cm) node[pos=1, font=\tiny, right] {$3$};
\draw[thick] (.2,0) arc (0:135:.2cm) node[pos=1, font=\tiny, left] {$3$};
\draw[thick] (.2,0) -- (.8,0) node[pos=1, font=\tiny, below] {$4$};
\draw[thick] (3.2142,2.4142) arc (180:315:.2cm) node[pos=0, font=\tiny, above] {$4$} node[pos=1, font=\tiny, right] {$5$};
\draw[thick] (-.5071,.7071) arc (0:-45:.2cm) node[pos=1, font=\tiny, left] {$5$};
\draw[thick] (-.5071,.7071) -- (1.5071,.7071);
\draw[thick] (1.5071,.7071) arc (180:0:.2cm);
\draw[thick] (1.9071,.7071) -- (3.9213,.7071);
\draw[thick] (3.9123,.7071) arc (180:225:.2cm) node[pos=1, font=\tiny, right] {$6$};
\end{tikzpicture}
\caption{The path $S_1$ on $V_4$.}
\label{fig:S_1}
\end{figure}
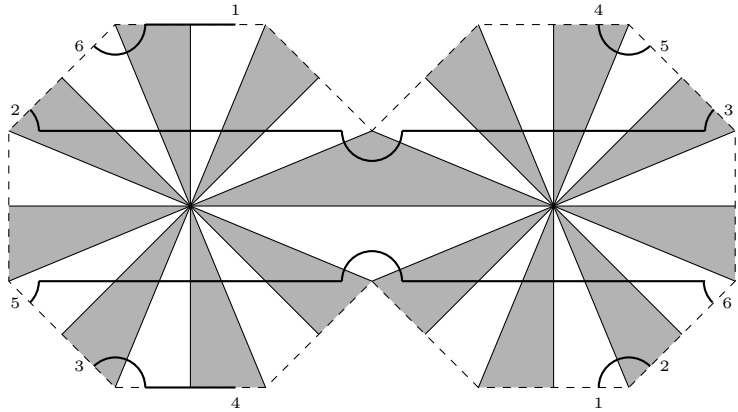

\begin{proof}[Proof of Theorem~\ref{thm:V_4}]
Again we may assume $f$ degenerates to a saddle connection on $V_4$ containing one of $s_1$ or $s_{-1}$. Assume it contains $s_1$.
Let $g = \phi^{-1}\mathring f$ for some $\phi\in \langle \sigma,\tau_e,\tau_o\rangle$ be a horizontal saddle connection with semicircles introduced. 
Suppose $g$ does not contain $L_{-3}$. Then starting from $L_3$, the sequence of components in $g$ before returning to $L_3$ must be either
$(L_1R_{-1})^jL_1$ or $(R_1L_{-1})^jR_1$ for some $j$. Since $i'$ interchanges $L_k$ and $R_k$, by Corollary~\ref{cor:three cases}
the opposite sequence must also occur. But since $f$ can only strike the midpoint of the base of $V_4$ twice, it can only contain two copies of
$L_3$, hence it must consist precisely of $L_3(L_1R_{-1})^jL_1L_3(R_1L_{-1})^jR_1$ or $L_3(R_1L_{-1})^jR_1L_3(L_1R_{-1})^jL_1$, which are
identical up to reparametrization. The homology class of this curve is $2\gamma_1+j(2\gamma_0-\beta_1-\beta_{-1})$.
By Proposition~9.9 of \cite{NearlyIsosc}, we have some odd integer $r$ such that
$(p^*_1-p^*_2)(\phi^*\gamma_1)\equiv r+n\mod 2n$, while $(p^*_1-p^*_2)(\phi^*(2\gamma_0-\beta_1-\beta_{-1}))\equiv 0\mod 2n$. Thus
$$(p^*_1-p^*_2)(\phi^*(2\gamma_1+j(2\gamma_0-\beta_1-\beta_{-1})))\equiv 2r\not\equiv 0\mod 2n$$
so $p(\mathring f)$ is not null-homologous, contradicting Lemma~\ref{lem:from stable}.

It follows that $g$ must contain $L_{-3}$. Thus by the same analysis, $g$ travels along $L_3(L_1R_{-1})^jL_{-3}(L_{-1}R_1)^j$
or $L_3(R_1L_{-1})^jL_{-3}(R_{-1}L_1)^j$, which are $S_j$ and $-S_j$ respectively.
Furthermore, since $g$ contains $s_{-1}$ as well, the $s_{-1}$ case is identical.
Thus $\mathring f = \pm\phi S_j$.
\end{proof}

\section*{Acknowledgments}

I would like to thank Howard Masur for helpful conversations on billiards,
MathOverflow user Matheus for directing me to \cite{FlatStruc} and correcting a misunderstanding of mine concerning the article,
and my longtime friend and mentor Bozenna Pasik-Duncan for her advice in the development of this manuscript as well as my mathematical
career as a whole.
I am also greatly indebted to Pat Hooper and Richard Schwartz for creating the McBilliards software, which was invaluable
for experimentation in the early stages of my research.

\end{document}